\documentclass{article}
\usepackage{amsmath,amsfonts,xcolor,amsthm,natbib,appendix,prodint}
\usepackage[margin=0.8in]{geometry}
\usepackage[hidelinks]{hyperref}
\usepackage[capitalize,nameinlink]{cleveref}
\usepackage{comment}
\usepackage{apptools}
\usepackage{authblk}

\usepackage{titletoc}
\newcommand\DoToC{%
  \startcontents
  \printcontents{}{1}{\textbf{Contents}\vskip3pt\hrule\vskip5pt}
  \vskip3pt\hrule\vskip5pt
}

\newcommand{\E}{\mathbb{E}}
\renewcommand{\P}{\mathbb{P}}
\newcommand{\F}{\mathcal{F}}
\newcommand{\R}{\mathbb{R}}

\newcommand{\ind}{\perp\!\!\!\!\perp}
\renewcommand{\d}{\mathrm{d}}
\newcommand{\V}{\mathbb{V}}
\newcommand{\cov}{\mathbb{V}}

\newcommand{\G}{\mathcal{G}}
\renewcommand{\H}{\mathcal{H}}

\newtheorem{prop}{Proposition}
\newtheorem{thm}{Theorem}
\newtheorem{lem}{Lemma}
\newtheorem{cor}{Corollary}

\AtAppendix{\counterwithin{asp}{subsection}}
\AtAppendix{\counterwithin{prop}{subsection}}
\AtAppendix{\counterwithin{thm}{subsection}}
\AtAppendix{\counterwithin{lem}{subsection}}
\AtAppendix{\counterwithin{cor}{subsection}}
\AtAppendix{\counterwithin{defn}{subsection}}

\title{On a theory of martingales for censoring}

\author[1,2]{Benjamin R. Baer\footnote{Corresponding author: benjamin.baer@st-andrews.ac.uk}}
\author[2]{Robert L. Strawderman}
\affil[1]{School of Mathematics and Statistics, University of St Andrews, St Andrews, Scotland, KY16 9SS}
\affil[2]{Department of Biostatistics and Computational Biology, University of Rochester, Rochester, NY, USA, 14642}
\begin{document}

\maketitle

\begin{abstract}
    A theory of martingales for censoring is developed. The Doob-Meyer martingale is shown to be inadequate in general, and a repaired martingale is proposed with a non-predictable centering term. Associated martingale transforms, variation processes, and covariation processes are developed based on a measure of half-predictability that generalizes predictability. The development is applied to study the Kaplan Meier estimator. 
\end{abstract}

\section{Introduction}
\label{sec:intro}

A great deal of work in survival analysis assumes that the failure and censoring distributions are absolutely continuous with respect to Lebesgue measure. Such assumptions are pragmatic, stemming primarily from the reasonable expectations of sufficient realism and the relative mathematical ease. 
Additionally, a substantial amount of methodological and theoretical work in survival analysis %
assumes that the failure and censoring distributions are discrete as a way to avoid delicate probabilistic arguments. Given these considerations, we might anticipate that classical survival analysis theory is well polished, especially in the discrete case. 

The seminal work of \cite{aalen78} brought the Doob-Meyer theory of martingales derived from counting processes to the forefront of methodological research in survival analysis; there has been a tremendous amount of work published since that time focusing on generalizations and extensions of Aalen's work. Especially influential book-length treatments of this still-active research area include \cite{fleming1991counting}, \cite{andersen1993statistical}, and \cite{kalbfleisch2011statistical}.

The purpose of this work is to reveal and repair critical shortcomings with the Doob-Meyer martingale theory surrounding the censoring distribution when there is no assumption of absolute continuity. Although the censoring distribution does not have a central focus in many areas of survival analysis, it arises in the study of estimation of functionals of the failure distribution. For example, in their seminal textbooks on efficiency theory for censored and missing data, \citet[Theorem 1.1]{van2003unified} and \citet[Section 9.3]{tsiatis2006semiparametric} calculate the censoring tangent space assuming discrete time and independent censoring. Results in later sections of this work reveal that each implicitly assumed that their failure and censoring distributions do not share discontinuities; indeed, their calculations are based on a stochastic process for censoring that only reduces to a martingale under this assumption but is generally only an approximation (cf.\ \cref{cor:marting-sharp-coerc}). 
In fact, the authors are only aware of
a few works that partially address the aforementioned deficiency are brief comments by  \citet[p. 36]{gill1980censoring},
\citet[p. 170]{bakry1994lectures}, and 
very recently some developments in
\citet{parner2023regression}. 
See also the recent work of \citet{baer2023recurrent} and \citet{strawderman2023volterra} which avoid relying on approximations that are only valid under a no-shared-discontinuity assumption; although martingale theory itself plays no direct role in these developments. 

In many respects, this work could have appeared decades ago, either as part of or shortly after the insightful developments in the monograph \emph{Censoring and Stochastic Integration} by \citet{gill1980censoring}; nevertheless the results are clearly still timely given that estimation in survival analysis remains an active area of research. 
\cref{sec:set-up} defines the random variables under study. \cref{sec:basic} sets the stage by constructing the Doob-Meyer martingale for censoring, developing a repaired martingale, and comparing the two. Interestingly, both martingales represented centered versions of the same counting process. 
Martingale transforms and variation and covariation processes with respect to the repaired martingale are studied in \cref{sec:martingale-trans}, 
then adapted in \cref{sec:pred} to produce familiar versions based on standard stochastic process theory and predictability. 
Three applications of the theory are presented in \cref{sec:apps} concerning 
estimators of the survival function. 
\cref{sec:disc} concludes with a discussion. 
All proofs are in the appendix.

\section{Definitions and Notation}
\label{sec:set-up}
\subsection{Observable and latent random variables}

Define the time at-risk $X>0$ and the failure indicator $\Delta \in \{0,1\}$, and
let these be jointly distributed random variables with distribution $\P$ and associated expectation $\E$. Assume $X<\infty$. 

Although not strictly necessary for most
developments, we shall also consider latent random variables $T,C>0$ satisfying  $X = T \wedge C$ and $\Delta=I(T \leq C)$ which represent the potential failure and censoring time, respectively. 
Denote the joint
distribution of these latent variables as $\P^*$ with associated expectation $\E^*;$ the observed data distribution $\P$ is then induced by the latent variable distribution $\P^*$. 
Unless otherwise specified, we do not assume $T$ and $C$ are independent. 

\subsection{Key definitions}

For $u > 0$, define the stochastic processes
\begin{align*}
    N_T(u) := I( X \leq u , \Delta=1 ), \hspace{5mm}
    N_C(u) & := I( X \leq u , \Delta=0 ),
\end{align*}
where $N_T$ (resp. $N_C$) is the observed failure (resp. censoring) process. 
Define $\delta$ as the increment operator so that $\delta N_C(u) = N_C(u) - N_C(u-)$. 
For $t \geq 0$, define the corresponding natural filtration 
\begin{equation*}
    \F_t = \sigma \{N_T(u), N_C(u) \,:\, 0 < u \leq t\}.
\end{equation*}

Next, define the stochastic processes
\begin{align*}
    Y^{\sharp}(u) := I\{ N_C(u-)=0, N_T(u-)=0\}, \hspace{5mm}
    Y^{\dagger}(u) &:= I\{ N_C(u-)=0, N_T(u)=0 \}. 
\end{align*}
For every $u > 0$, 
the usual at-risk process is given by
$Y^{\sharp}(u) = I(X \geq u);$ the
modified at-risk process $Y^{\dagger}$
is instead given by
\[
Y^{\dagger}(u) \, = \, Y^\sharp(u) -
\delta N_T(u) \, = \, 
Y^\sharp(u+) + \delta N_C(u).
\]
As expected, $Y^{\sharp}$ is left-continuous; 
however, the modified at-risk process $Y^{\dagger}$ is neither left-continuous nor right-continuous.

The process $Y^{\dagger}$ was introduced in \citet{baer2023recurrent} and plays an important role throughout this paper. Expressed in terms of the latent variables, the modified at-risk process $Y^{\dagger}(u) = I(T>u, C\geq u);$ hence, at time $u>0$, this process can be informally interpreted as capturing the risk of being censored at $u$ when 
the risk for failure has ``passed'' (i.e., lies beyond $u$). It will become clear in later developments that
$Y^{\dagger}$ plays
an important role in the representation and
estimation of functionals of interest in survival analysis. %

\section{Basic Theoretical Developments}
\label{sec:basic}

In this section we construct the Doob-Meyer martingale for censoring, develop a repaired martingale, and compare the two. 

\subsection{Doob-Meyer martingales}

Define the cause-specific cumulative functions 
\begin{align*}
    \Lambda_T^{\sharp}(t) := \int_{(0, t]} \frac{\d \E \{N_T(u)\} }{\E\{Y^{\sharp}(u)\}}, \hspace{5mm}
    \Lambda_C^{\sharp}(t) & := \int_{(0, t]} \frac{\d \E \{N_C(u)\} }{\E\{Y^{\sharp}(u)\}},
\end{align*}
which are well defined for all $t>0$ through the usual $0/0=0$ convention. 
Finally, define for $t \geq 0$ the familiar stochastic processes 
\begin{align*}
    M_T^{\sharp}(t) := N_T(t) - \int_{(0,t]} Y^{\sharp}(u) \,\d \Lambda_T^{\sharp}(u), \hspace{5mm}
    M_C^{\sharp}(t) & := N_C(t) - \int_{(0,t]} Y^{\sharp}(u) \,\d \Lambda_C^{\sharp}(u).
\end{align*}
The following proposition states a result that is well-known.
\begin{prop}
\label{prop:sharp-martingale}
    $M_T^{\sharp}$ and $M_C^{\sharp}$ are the unique Doob-Meyer $\F_t$-martingales.
\end{prop}

Since $M_C^{\sharp}$ is the unique Doob-Meyer martingale constructed by centering the observed censoring process $N_C$ with a $\F_t$-predictable compensator, $M_C^{\sharp}$ is often taken to be the relevant martingale when studying estimation of the censoring distribution and related functionals under random censoring. However, as will be shown below, this choice is not always appropriate.

\subsection{Identification in the latent variable model}

Define the latent marginal 
cumulative hazards as
\begin{equation*}  
    \Lambda_T(t) := \int_{(0,t]} \frac{\mathrm{d}\P^*(T \leq u)}{\P^*(T \geq u)}, \hspace{5mm}
    \Lambda_C(t) := \int_{(0,t]} \frac{\d \P^*(C \leq u)}{\P^*(C \geq u)}. 
\end{equation*}
These are functionals of the latent variable distribution $\P^*$ and cannot necessarily be estimated with the observed data in the absence of assumptions. 

Under independent censoring (i.e., $T \ind C$), it is well-established that
$\Lambda_T^{\sharp}(t) = \Lambda_T(t)$ for all $t>0$ of practical interest. %
Given the form of the Doob-Meyer martingale $M_C^{\sharp}$, one might assert that $\Lambda_C^{\sharp}(t) = \Lambda_C(t)$ for such $t$ under independent censoring. 
However, this is not necessarily true. Define
\begin{equation*}
    \Lambda_C^{\dagger}(t) := \int_{(0, t]} \frac{\d \E \{N_C(u)\} }{\E\{Y^{\dagger}(u)\}};
\end{equation*}
similarly to $\Lambda_C^{\sharp}$, the function $\Lambda_C^{\dagger}$
is well-defined as a Lebesgue-Stieltjes integral under the $0/0=0$ convention; 
see \cref{sec:supp:basic-addntl:props}, where it is additionally
shown that $\Lambda_C^{\dagger}$ is  right-continuous. %
The result given below establishes that
$\Lambda_C^{\dagger}$ is the identification of $\Lambda_C$ under independent censoring. 
\begin{thm}
\label{prop:identif}
    If $T \ind C$, then $\Lambda_T^{\sharp}(t) = \Lambda_T(t)$ 
    for $t>0$ satisfying $\E \{Y^{\sharp}(t)\}>0$ %
    and $\Lambda_C^{\dagger}(t) = \Lambda_C(t)$ for $t>0$ satisfying $\E \{Y^{\dagger}(t)\}>0$. %
\end{thm}
\noindent The asymmetry in identification
of the cumulative hazards for failure 
and censoring stem directly from the asymmetry in $\Delta$. This same theorem was proved
in \citet{baer2023recurrent} under a nominally weaker assumption. The equality does not necessarily hold for large $t$ since for example aggressive early censorship could preclude observation of late failures. 

Additional assumptions are needed to be
able to equate $\Lambda_C$ and $\Lambda_C^{\sharp}$. 
In view of \cref{prop:identif}, 
these assumptions follow from
the assumptions needed to be able to
equate $\Lambda_C^{\dagger}$ and $\Lambda_C^{\sharp}$. 
Considering their definitions and where all quantities are finite, the latter are directly related through the expression
\begin{equation*}
    \d \Lambda_C^{\sharp}(u)
    = \frac{\E\{Y^{\dagger}(u)\}}{\E\{Y^{\sharp}(u)\}} \d \Lambda_C^{\dagger}(u)
    = \left\{ 1- \frac{\delta \E \{ N_T(u) \}}{\E\{Y^{\sharp}(u)\}} \right\} \d \Lambda_C^{\dagger}(u) =
    \left\{ 1- \delta \Lambda^\sharp_T(u) \right\} \d \Lambda_C^{\dagger}(u).
\end{equation*}
Evidently, 
equality of $\Lambda_C^{\sharp}$
and $\Lambda_C^{\dagger}$ therefore depends on
the behavior of both
$\Lambda_C^{\dagger}$ and $\Lambda_T^{\sharp}$.
For example, the identity above shows
that $\d \Lambda_C^{\sharp}(u) = \d \Lambda_C^{\dagger}(u)$ 
when $\Lambda^\sharp_T$ is continuous at $u$. 
The next result establishes necessary and sufficient conditions
under which $\Lambda_C^{\sharp} = \Lambda_C^{\dagger}$.
\begin{prop}
\label{cor:lambda-sharp-coerc}
    For any $t>0$, $\Lambda_C^{\sharp}(t) = \Lambda_C^{\dagger}(t)$ if and only if $\delta \Lambda_T^{\sharp}(u)
    \delta \Lambda_C^{\dagger}(u) = 0$ for all $u \in (0,t]$. 

\end{prop}

An obvious corollary is, under independent censoring and for $t>0$ satisfying $\E \{Y^{\dagger}(t)\}>0$, that $\Lambda_C^{\sharp}(t)=\Lambda_C(t)$ if and only if $\delta \E\{N_T(u)\} \delta \E\{N_C(u)\}=0$ for all $u \in (0,t]$. This condition expresses that there are no-shared-discontinuities between the failure and censoring distributions and is entirely avoided by studying $\Lambda_C^{\dagger}$ rather than $\Lambda_C^{\sharp}$.

\subsection{A repaired censoring martingale}

Define the stochastic process 
\begin{align*}
    M_C^{\dagger}(t) 
    & := N_C(t) - \int_{(0,t]} Y^{\dagger}(u) \,\d \Lambda_C^{\dagger}(u)
    = N_C(t) - \int_{(0,t]} Y^{\dagger}(u) \frac{\d \E \{N_C(u)\} }{\E\{Y^{\dagger}(u)\}}. 
\end{align*}
The process $M_C^{\dagger}$ was recently studied by \citet{baer2023recurrent} in a nonparametric efficiency theory setting; it vanishes in expectation due to the appearance of $Y^{\dagger}$ in both the numerator and denominator. The following result shows that it is also a martingale.
\begin{thm}
\label{prop:dagger-martingale}
    $M_C^{\dagger}$ is an $\F_t$-martingale.
\end{thm}

At first glance, this theorem would appear to contradict the uniqueness noted in \cref{prop:sharp-martingale}, which may be regarded as a direct consequence of the Doob-Meyer theorem. However, it is important to keep in mind that the Doob-Meyer theorem only pertains to centering terms (i.e. compensators) with finite expectation that are increasing, right-continuous, and 
${\cal F}_t$-predictable. The following result establishes that 
there is no contradiction to the Doob-Meyer theorem
because the centering term
\begin{equation*}
    A_C^{\dagger}(t) 
    := \int_{(0,t]} Y^{\dagger}(u) \frac{\d \E \{N_C(u)\} }{\E\{Y^{\dagger}(u)\}}
\end{equation*}
used to define $M_C^{\dagger}$ is not ${\cal F}_t$-predictable, hence does not meet the key premise of the Doob-Meyer theorem. 

\begin{prop}
\label{lem:mcd-properties}
    The centering term $A_C^{\dagger}$ in the $\F_t$-martingale
    $M_C^{\dagger} = N_C - A_C^{\dagger}$ is increasing, right-continuous, satisfies $A^{\dagger}_C(0)=0$ but is not $\F_t$-predictable. 
\end{prop}

Although the result makes clear that $A_C^{\dagger}$ is a non-standard centering term due to its non-predictability, it does not clarify the relationship between $A_C^{\dagger}$ and the $\F_t$-predictable centering term of the Doob-Meyer martingale $M_C^{\sharp}$. When $X$ is discrete we may readily see for $t \geq 0$ that the centering term of $M_C^{\sharp}$ satisfies that 
\begin{equation}
\label{eq:doob-sharp}
    A_C^{\sharp}(t)
    := \int_{(0,t]} Y^{\sharp}(u) \,\frac{\d \E \{N_C(u)\}}{\E\{Y^{\sharp}(u)\}} 
    = \sum_{0 < u \leq t} \E \{ \delta N_C(u) \mid \F_{u-} \}
\end{equation}
is a sum of expectations conditional on $\F_{u-}$. Indeed, this is an example of the Doob decomposition \citep[Theorem 5.2.10]{durrett2010probability} for discrete time martingales, and it is routinely used to motivate properties of Doob-Meyer martingales in continuous time. 

Define the filtration
\begin{equation*}
    \G_t
    := \sigma \{N_T(u) \,:\, 0 < u \leq t; \hspace{1mm} 
    N_C(u) \,:\, 0 < u < t\}.
\end{equation*}
Notice that $\F_{t-} \subseteq \G_t \subseteq \F_t$ for all $t>0$: whereas $\F_{t-}$ contains all information \emph{before} time $t$, the sigma algebra $\G_t$ additionally contains \emph{present} information about $N_T$ at time $t$. 
When $X$ is discrete we may readily see for $t \geq 0$ that
\begin{equation}
\label{eq:doob-dagger}
    A_C^{\dagger}(t)
    = \sum_{0 < u \leq t} \E \{ \delta N_C(u) \mid \G_{u} \}
\end{equation}
is a sum of expectations conditional on $\G_{u}$ rather than on $\F_{u-}$. 
Importantly, the martingale ``drift-free'' property holds for both $M_C^{\sharp}$ and $M_C^{\dagger}$ by removing information in the past.

\subsection{Comparison}

The identification result in \cref{prop:identif}, combined
with \cref{prop:dagger-martingale}, shows that $M_C^{\dagger}$ is 
in general the relevant martingale to analyze when the latent cumulative hazard $\Lambda_C$ for censoring is of interest. 
We now state a result that directly 
compares $M_C^{\dagger}$ to the Doob-Meyer martingale $M_C^{\sharp};$
this may be considered as a modest extension of \cref{cor:lambda-sharp-coerc}. Recall that two stochastic processes are indistinguishable if they are almost surely equal uniformly over time. 
\begin{cor}
\label{cor:marting-sharp-coerc}    
    $M_C^{\sharp}$ and $M_C^{\dagger}$ are indistinguishable if and only if $\delta \Lambda_T^{\sharp}(u)
    \delta \Lambda_C^{\dagger}(u) = 0$ for all $u>0$. %
    \end{cor}
\noindent Under the more stringent no-shared-discontinuities condition in \cref{cor:marting-sharp-coerc}, the Doob-Meyer martingale $M_C^{\sharp}$ can be studied
in place of $M_C^{\dagger}$. Earlier work has often implicitly relied on this no-shared-discontinuities condition by working with $M_C^{\sharp}$ \citep[see e.g.][]{van2003unified,tsiatis2006semiparametric}.

\section{Martingale Transforms}
\label{sec:martingale-trans}

We now turn our attention to developing a theory for transforms of the martingale $M_C^{\dagger}$. Throughout we exploit that the centering term $A_C^{\dagger}$ of the $\F_t$-martingale $M_C^{\dagger}$ is $\G_t$-adapted. We start by generalizing the notion of $\F_t$-predictability. 
Define the set of half-predictable rectangles 
\begin{equation}
\label{eq:half-pred-rect}
    \{0\} \times A, A \in \G_0 \hspace{2mm} \text{ and } \hspace{2mm} [a,b) \times A, 0 < a < b \leq \infty, A \in \G_a.
\end{equation}
Define $\H$ as the sigma algebra generated by the half-predictable rectangles, and define a process $H$ as half-predictable if, as a mapping from $[0, \infty) \times \Omega$ to $\R$, it is $\H$-measurable. 

\begin{lem}
\label{lem:half-pred}
    If $H$ is a half-predictable process, then $H$ is $\F_t$-predictable and, for any $t>0$, $H(t)$ is $\G_t$-measurable. 
\end{lem}

The lemma show that half-predictability generalizes $\F_t$-predictability and implies $\G_t$-adaptedness.

\subsection{Martingale transforms}

Define the martingale transform
\begin{equation*}
    (H \boldsymbol{\cdot} M^{\dagger}_{C}) (t) 
    := \int_{(0,t]} H(u) \, \d M_C^{\dagger}(u),
\end{equation*}
for a stochastic process $H$. \citet[Theorems 1.5.1 and 2.4.1]{fleming1991counting} show that $H \boldsymbol{\cdot} M^{\sharp}_{C}$ is a martingale whenever $H$ is $\F_t$-predictable and bounded, driven by the predictability of the centering term (or compensator) of $M^{\sharp}_C$.  

For the martingale of interest $M_C^{\dagger}$, the centering term is not $\F_t$-predictable but is $\G_t$-adapted. The following result is suggested by the corresponding heuristic motivation that for all $u>0$ 
\begin{equation*}
\label{eq:mt:no-inc}
    \E \{ \delta M_C^{\dagger}(u) \mid \G_u \}
    = 0, 
\end{equation*}
which is established in detail in \cref{sec:supp:mt-addntl:proofs}. 

\begin{thm}
\label{thm:martingale-transform}
 If $H$ is half-predictable and bounded, then $H \boldsymbol{\cdot} M^{\dagger}_{C}$ is a $\F_t$-martingale. 
\end{thm}

Since the constant function $u \mapsto 1$ is $\G_t$-adapted and bounded, \cref{thm:martingale-transform} shows that the stochastic process $M^{\dagger}_{C} = 1 \boldsymbol{\cdot} M^{\dagger}_{C}$ is a $\F_t$-martingale. Thus \cref{thm:martingale-transform} generalizes \cref{prop:dagger-martingale} to create a large class of martingales based on $M_C^{\dagger}$. 
A simplified statement and proof of \cref{thm:martingale-transform} is provided in \cref{sec:supp:mt-addntl:proofs-discrete} that assumes discrete time.

\subsection{Variation and covariation processes}

Second order moments of martingale transforms may be conveniently calculated based on a theory of covariation processes. 
\citet[Theorems 2.4.3 and 2.6.1]{fleming1991counting} study predicable covariation processes based on $M_T^{\sharp}$; 
in this section we consider covariation processes based on $M_C^{\dagger}$ that have half-predictable centering terms.

Define the multiplicative survival functions 
\begin{equation*}
    F^{\sharp}(t) := \Prodi_{(0,t]} \biggl\{ 1 - \d \Lambda_T^{\sharp}(u) \biggr\}, \hspace{5mm}
    G^{\dagger}(t) := \Prodi_{(0,t]} \biggl\{ 1 - \d \Lambda_C^{\dagger}(u) \biggr\},
\end{equation*}
where $\prodi$ is the product integral \citep{gill1990survey}. 

\begin{thm}
\label{thm:covar-proc}
     If $H_1, H_2$ are bounded and half-predictable processes, then the process 
    \begin{equation*} 
        (H_1 \cdot M_C^{\dagger})(H_2 \cdot M_C^{\dagger}) - \int_{(0,\cdot]} H_1(u) H_2(u) \frac{G^{\dagger}(u)}{G^{\dagger}(u-)} Y^{\dagger}(u) \,\d \Lambda_C^{\dagger}(u) 
    \end{equation*} 
    is an $\F_t$-martingale. 
    If $H_1$ is a bounded and half-predictable process and $H_3$ is a bounded and $\F_t$-predictable process, then the process 
    \begin{equation*} 
        (H_1 \cdot M_C^{\dagger})(H_3 \cdot M_T^{\sharp}) 
    \end{equation*} 
    is an $\F_t$-martingale. 
\end{thm}

The theorem shows that $M^{\sharp}_{T}$ and $M^{\dagger}_{C}$ are orthogonal despite that $M^{\sharp}_{T}$ and $M^{\sharp}_{C}$ are generally only orthogonal when $\E\{N_T(\cdot)\}$ and $\E\{N_C(\cdot)\}$ share no discontinuities (see e.g. Theorem 2.6.1 in \citet{fleming1991counting}). A simplified statement and proof of \cref{thm:covar-proc} is provided in \cref{sec:supp:mt-addntl:var-and-covar-discrete} that assumes discrete time. 

    The covariation process result could be anticipated by considering $H_1(u) = H_3(u) = I(u=a)$, for some $a \in (0, t]$. In this case the product of the two martingale transforms
    \begin{align*}
        (H_1 \boldsymbol{\cdot} M^{\dagger}_{C}) (t) \, (H_3 \boldsymbol{\cdot} M^{\sharp}_{T}) (t) 
        & = \left\{ \delta N_C(a) - Y^{\dagger}(a) \delta \Lambda_C^{\dagger}(a) \right\} \left\{ \delta N_T(a) - Y^{\sharp}(a) \delta \Lambda_T^{\sharp}(a) \right\} \\
        & \hspace{-15mm} = \left\{ \delta N_C(a) - Y^{\dagger}(a) \delta \Lambda_C^{\dagger}(a) \right\} \left\{ - \delta \Lambda_T^{\sharp}(a) \right\}
    \end{align*}
    is a martingale by \cref{thm:martingale-transform} as it is proportional to $(H_1 \boldsymbol{\cdot} M^{\dagger}_{C}) (t)$. In this simple case, a half-predictable centering term is identically zero.

\section{Alternative Development: Doob-Meyer Martingales and Predictability}
\label{sec:pred}

The development so far has been based entirely on $\G_t$-adaptedness due to the special structure of the centering term $A_C^{\dagger}$ of $M_C^{\dagger}$. In this section, we exploit standard stochastic process theory from the Strasbourg school \citep[p. 116-117]{andersen1993statistical} to develop alternative versions of some of the previous results which make more stringent assumptions. 
The starting point is the following algebraic identity, which expresses $M_C^{\dagger}$ through the Doob-Meyer martingales $M_T^{\sharp}$ and $M_C^{\sharp}$. 
\begin{lem}
\label{lem:dag-dm}
    For all $t>0$, $M_C^{\dagger}(t) = M^\sharp_C(t) + \int_{(0,t]} \delta \Lambda^\dag_C(u) \,\d M_T^{\sharp}(u)$. 
\end{lem}
\noindent The identity provides an alternative proof of \cref{prop:dagger-martingale} which states that $M_C^{\dagger}$ is a martingale, as described in detail in \cref{sec:supp:pred-addntl:alt-proofs}. An alternative proof of \cref{cor:marting-sharp-coerc} is also provided there. 

Now we consider a version of \cref{thm:martingale-transform} which concerns transforms of the martingale $M_C^{\dagger}$. In light of \cref{lem:dag-dm}, such martingale transforms are simply sums of transforms of $M_C^{\sharp}$ and $M_C^{\dagger}$, which are well understood. 
\begin{prop}
\label{prop:pred-trans}
     Assume $\E [ \int_{(0,\infty)} H^2(u) \{ 1 - \delta \Lambda_T^{\sharp}(u) \} \,\d \Lambda_T^{\sharp}(u) ] < \infty$ and $\E [ \int_{(0,\infty)} H^2(u) \{ 1 - \delta \Lambda_C^{\sharp}(u) \} \,\d \Lambda_C^{\sharp}(u) ] < \infty$. If $H$ is $\F_t$-predictable and locally bounded, then $H \boldsymbol{\cdot} M^{\dagger}_{C}$ is a $\F_t$-martingale.
\end{prop}

\noindent Compared with \cref{thm:martingale-transform}, this result assumes that $H$ is $\F_t$-predictable, which is more restrictive than being half-predictable.

We now consider variation and covariation processes of martingale transforms. 
\begin{prop}
\label{prop:pred-covar}
    If $H_1$ and $H_2$ are $\F_t$-predictable and locally bounded, then the following is a local $\F_t$-martingale 
    \begin{equation*}
        (H_1 \boldsymbol{\cdot} M^{\dagger}_{C})(t) \, (H_2 \boldsymbol{\cdot} M^{\dagger}_{C})(t) - \int_{(0,t]} H_1(u) H_2(u) \frac{F^{\sharp}(u) G^{\dagger}(u)}{F^{\sharp}(u-) G^{\dagger}(u-)} Y^{\sharp}(u) \,\d \Lambda_C^{\dagger}(u). \label{eq:pred-var}
    \end{equation*}
    If $H_1$ and $H_2$ are $\F_t$-predictable and locally bounded, then the following is a local $\F_t$-martingale 
    \begin{equation*}
        (H_1 \boldsymbol{\cdot} M^{\dagger}_{C})(t) \, (H_2 \boldsymbol{\cdot} M^{\sharp}_{T})(t).
    \end{equation*}
\end{prop}

The centering terms in \cref{prop:pred-covar} are $\F_t$-predictable as a direct consequence of applying the standard theory to provide them. In the following corollary, we leverage another algebraic identity analogous to \cref{lem:dag-dm} and stated in \cref{sec:supp:pred-addntl:proofs} to produce the same half-predictable centering terms as in \cref{thm:covar-proc}.

\begin{cor}
\label{thm:pred-covar-proc}
    If $H_1$ and $H_2$ are $\F_t$-predictable and locally bounded, then the following is a local $\F_t$-martingale 
    \begin{equation*}
        (H_1 \boldsymbol{\cdot} M^{\dagger}_{C})(t) \, (H_2 \boldsymbol{\cdot} M^{\dagger}_{C})(t) - \int_{(0,t]} H_1(u) H_2(u) \frac{F^{\sharp}(u)}{F^{\sharp}(u-)} Y^{\dagger}(u) \,\d \Lambda_C^{\dagger}(u). 
    \end{equation*}
    If $H_1$ and $H_2$ are $\F_t$-predictable and locally bounded, then the following is a local $\F_t$-martingale 
    \begin{equation*}
        (H_1 \boldsymbol{\cdot} M^{\dagger}_{C})(t) \, (H_2 \boldsymbol{\cdot} M^{\sharp}_{T})(t).
    \end{equation*}
\end{cor}

Although the development in this section is fairly intuitive, it is also somewhat less natural than the earlier development based on half-predictability. Besides requiring a stronger assumption of $\F_t$-predictability, the results in this section break down in more general contexts; %
see \cref{sec:apps:class} for more discussion. 

\section{Applications to the Kaplan Meier Estimator}
\label{sec:apps}

In this section we highlight two applications of the theory developed for $M_C^{\dagger}$. 
Among the most pervasive problems in survival analysis is the estimation of the failure survival function 
\begin{equation*}
    F(t) 
    := \Prodi_{(0,t]} \biggl\{ 1 - \d \Lambda_T(u) \biggr\}.
\end{equation*}
based on an i.i.d. sample of the observed data with size $n$. Denote $\P_n$ as the empirical distribution with expectation $\E_n$. 

Since $F(t) = F^{\sharp}(t)$ under independent censoring and for $t>0$ satisfying $\E \{ Y^{\sharp}(t) \}>0$, by \cref{prop:identif}, this problem is often reduced to estimation of $F^{\sharp}(t)$. The standard nonparametric estimator is the \citet{kaplan1958nonparametric} estimator 
\begin{align*}
    F^{\sharp}_n(t)
    := \Prodi_{(0,t]} \biggl\{ 1 - \frac{\d \E_n \left\{ N_T(u) \right\}}{\E_n \{ Y^{\sharp}(u) \}} \biggr\}, 
\end{align*}
which may be recognized as a plugin estimator given the form of $\Lambda_T^{\sharp}$ and the definition of $F^{\sharp}$. 

The martingale theory in the previous sections was developed for a single observation. Throughout this section we apply minor extensions of these earlier results to a sample of independent observations. %

\subsection{Application I: Variance}

The theory of transforms of $M_C^{\dagger}$ may be applied to calculate the asymptotic variance of the Kaplan Meier estimator since its influence function 
\begin{equation}
    \frac{I(X>t)}{G^{\dagger}(t)} + \int_{(0,t]} \frac{F^{\sharp}(t)}{F^{\sharp}(u)} \frac{\d M_C^{\dagger}(u)}{G^{\dagger}(u)} - F^{\sharp}(t), \label{eq:km-if1}
\end{equation}
which dictates its first-order asymptotic behavior, contains a martingale transform based on $M_C^{\dagger}$.

\begin{prop}
\label{prop:km-var}
    Let $t>0$ and define the variance
    \begin{equation*}
        \sigma^2(t) := \{ F^{\sharp}(t) \}^2 \int_{(0,t]} \frac{1}{G^{\dagger}(u-)} \,\d \left\{ \frac{1}{F^{\sharp}(u)} \right\}. 
    \end{equation*}
    The Kaplan Meier estimator $F^{\sharp}_n(t)$ satisfies $\sqrt{n} \{ F^{\sharp}_n(t) - F^{\sharp}(t) \} \rightsquigarrow \mathcal{N}\{ 0, \sigma^2(t) \}$. 
\end{prop}

The asymptotic variance was first derived by \citet{greenwood1926report} in a rather different setting. The results of nonparametric efficiency theory show that $\sigma^2(t)$ is the nonparametric efficiency bound for estimating $F^{\sharp}(t)$ and hence that the Kaplan Meier estimator $F^{\sharp}_n(t)$ is asymptotically efficient \citep{bickel1993efficient}.

\subsection{Application II: Covariance}

The theory of transforms of $M_C^{\dagger}$ may again be applied to calculate the asymptotic covariance of the Kaplan Meier estimator at distinct time points. The influence function of the Kaplan Meier estimator, reported in \cref{eq:km-if1}, may be algebraically rewritten using integration by parts as
\begin{equation}
    - \int_{(0,t]} \frac{F^{\sharp}(t)}{F^{\sharp}(u)} \frac{\d M_T^{\sharp}(u)}{G^{\dagger}(u-)}, \label{eq:km-if2}
\end{equation}
as discussed by \citet[Section 5.1.3]{baer2023recurrent}. 

We may therefore calculate the covariance using the covariation process between transforms of $M_C^{\sharp}$ and $M_C^{\dagger}$. 

\begin{prop}
\label{prop:cox-km-covar}
    Let $s,t>0$ and define the covariance
    \begin{equation*}
        \sigma^2(s, t) := F^{\sharp}(s) F^{\sharp}(t) \int_{(0,s \wedge t]} \frac{1}{G^{\dagger}(u-)} \,\d \left\{ \frac{1}{F^{\sharp}(u)} \right\}. 
    \end{equation*}
    The Kaplan Meier estimator $F^{\sharp}_n$ satisfies that the vector $\Bigl( F^{\sharp}_n(s), F^{\sharp}_n(t) \Bigr)$ is asymptotically normal with expectation zero and covariance $\sigma^2(s,t)$. 
\end{prop}

Although the result derived here considers only two time points $s,t>0$, existing results show that the function $F^{\sharp}_n$ is asymptotically a Gaussian process with kernel $\sigma^2$ \citep{andersen1993statistical}.

\subsection{Application III: the class of estimators}
\label{sec:apps:class}

In this section, we consider estimation of the marginal functional
\begin{equation*}
    F^{\sharp}(t)
    = \Prodi_{(0,t]} \left\{ 1 - \frac{\d \E \{ N_T(u) \}}{\E \{Y^{\sharp}(u) \}} \right\}
\end{equation*}
when more data is available and other estimators are available besides the Kaplan Meier estimator. 

Consider that data on a finite-dimensional time-varying variable $L$ is available in addition to the data on $N_T, N_C$ considered previously. Denote $\bar{L}(t) = \{ L(u) \, : \, 0 \leq u \leq t\}$ as the history of the variable. With this additional data, the natural filtration is now $\F_t := \sigma\{ L(u), \, 0 \leq u \leq t ; \hspace{1mm} N_T(u), N_C(u), 0 < u \leq t \}$ and also $\G_t := \sigma\{L(u), 0 \leq u \leq t; \hspace{1mm} N_T(u) \,:\, 0 < u \leq t; \hspace{1mm} N_C(u) \,:\, 0 \leq u < t\}$. The results in \cref{sec:basic,sec:martingale-trans} may be readily verified as continuing to hold in this setting. 

The following result gives the influence function of all regular and asymptotically linear estimators for $F^{\sharp}(t)$, which characterizes all well-behaved (i.e. regular and asymptotically linear) estimators \citep{tsiatis2006semiparametric}. As with $F^{\sharp}$, the definitions of $G^{\dagger}$ and $M_C^{\dagger}$ are unchanged. 

\begin{prop}[Theorem 3 in \citet{baer2023recurrent}]
\label{prop:class}
    If $G^{\dagger}$ is known, the class of influence function for estimating $F^{\sharp}(t)$ is
    \begin{equation*}
        \frac{\Delta}{G^{\dagger}(X-)} I(X>t) + \int_{(0,\infty)} H\{u; \bar{L}(u)\} \,\d M_C^{\dagger}(u)  - F^{\sharp}(t),
    \end{equation*}
    where $H$ is sufficiently regular index function that varies to generate the class. 
\end{prop}
\noindent Note that the influence function of the Kaplan-Meier estimator belongs to this class; see Lemma 1 in \citet{baer2023recurrent}. 

When $L$ is left-continuous, it is $\F_t$-predictable, so the integral in \cref{prop:class} may be shown to be a martingale transform using the predictability-based techniques in \cref{sec:pred}. More generally, it may not be left-continuous; for example, \citet{baer2023recurrent} develop \cref{prop:class} in the context of $L$ being a recurrent event process which is clearly not $\F_t$-predictable. In general, the theory developed in \cref{sec:martingale-trans} may be used to study the influence functions.

\section{Discussion}
\label{sec:disc}

Although the basic survival analysis setting involving the time at-risk $X$ and the failure indicator $\Delta$ has been intensely studied, the vast majority of existing work has exclusively focused on the observed failure process $N_T$. We showed that the observed censoring process $N_C$ arises naturally in the study of $N_T$ and, surprisingly, that proper analysis differs in some important ways from that of $N_T$. The familiar terms $Y^{\sharp}, \Lambda_T^{\sharp}, M_T^{\sharp}$ are appropriate for the direct study of $N_T$ while the somewhat unfamiliar terms $Y^{\dagger}, \Lambda_C^{\dagger}, M_C^{\dagger}$ are most appropriate for the study of $N_C$. We developed an understanding of $M_C^{\dagger}$ through half-predictability that lead to general results on martingale transforms and covariation processes. 

The relationship between $\Delta = I(T \leq C)$ and the latent variables $T,C$ 
imposes that that censoring occurs after failure, even when coincident \citep{kaplan1958nonparametric}; as a pivotal consequence, ties (i.e. that $T=C$) are not considered to be observable. This may be updated by defining the failure indicator instead by $\Delta = I(T \leq C) + I(T = C) \in \{0,1,2\}$ \citep{langberg1978converting}. Equivalently, the observed failure and censoring processes could be defined as
\begin{equation*}
    N_T^{'}(u) := I( T \leq u , T \leq C ), \hspace{5mm}
    N_C^{'}(u) := I( C \leq u , C \leq T );
\end{equation*}
in this case, both jump at a tie, and $(N_C^{'}, N_T^{'})$ is no longer a multivariate counting process. Alternatively, a third process could be introduced that only jumps at ties.

A natural extension of our work would be to competing risks where multiple types of failure may coincide with each other. 
Additionally, extending the study of half-predictability to incorporate localization would interesting.

\section*{Acknowledgements}

The authors thank Ashkan Ertefaie for helpful conversations during the initial stages of the work. %
BB was partially supported by the National Institute of Neurological Disorders and Stroke (R61/R33 NS120240).

\bibliographystyle{apalike}
\bibliography{refs}

\newpage
\appendix

\DoToC

\section{Additional Material for \texorpdfstring{\cref{sec:basic}}{Section 2}}
\label{sec:supp:basic-addntl}

\subsection{Basic properties of \texorpdfstring{$\Lambda_C^{\dagger}$}{Lambda C dagger}}
\label{sec:supp:basic-addntl:props}

We start by showing that $\Lambda_C^{\dagger}$ is well-defined. First, we establish that $\E \{ N_C(t) - N_C(s) \} > 0$ implies that $\E \{ Y^{\dagger}(s) \} > 0$ for all $t > s>0$. According to the definitions of $N_C$ and $Y^{\dagger}$, we see that $\E \{ N_C(t) - N_C(s) \} = \P (\Delta=0, s < X \leq t)$ and $\E \{ Y^{\dagger}(s) \} = \P(X > s) + \P(X=s, \Delta=0)$. Since $\P (\Delta=0, s < X \leq t) \leq \P (\Delta=0, X > s)$, it is now easy to see that $\E \{ Y^{\dagger}(s) \} > 0$.
Next, we show that $\E \{ Y^{\dagger}(s) \} = 0$ implies 
$\E \{ N_C(t) - N_C(s) \} = 0$ for all $t > s>0$. Here,
$\E \{ Y^{\dagger}(s) \} = 0$ necessarily implies 
$\P (\Delta=0, X > s) = 0$; hence, we must also have
$\E \{ N_C(t) - N_C(s) \} = \P (\Delta=0, s < X \leq t) = 0$.
Under the usual convention that 0/0 = 0, 
it follows that $\Lambda_C^{\dagger}(t) < \infty$ is finite for $t \in {\cal S}$ and constant and well-defined for $t \in {\cal S}^c$.

Finally, we show that $\Lambda_C^{\dagger}$ is right-continuous. In
particular, note that
\begin{align*}
    \lim_{s \searrow t} \int_{(0, s]} \frac{\d \E \{N_C(u)\} }{\E\{Y^{\dagger}(u)\}}
    & = \lim_{s \searrow t} \int_{(0, \infty)} I(u \leq s) \frac{\d \E \{N_C(u)\} }{\E\{Y^{\dagger}(u)\}} \\
    & = \int_{(0, \infty)} \left\{ \lim_{s \searrow t} I(u \leq s) \right\} \frac{\d \E \{N_C(u)\} }{\E\{Y^{\dagger}(u)\}} \\
    & = \int_{(0, \infty)} I(u \leq t) \frac{\d \E \{N_C(u)\} }{\E\{Y^{\dagger}(u)\}} \\
    & = \int_{(0, t]} \frac{\d \E \{N_C(u)\} }{\E\{Y^{\dagger}(u)\}},
\end{align*}
where the second equality follows from the monotone convergence theorem.

\subsection{Proofs}
\label{sec:supp:basic-addntl:proofs}

Throughout the proofs in this section we adopt the following helpful notation to study tail behavior. Define the sets ${\cal S}^- = \{s: \E[Y^{\sharp}(s+)] >0 \}$, 
${\cal S} = \{s: \E[Y^{\dagger}(s)] >0 \}$,
and ${\cal S}^+ = \{s: \E[Y^{\sharp}(s)] >0 \}$. Then, because 
$Y^{\sharp}(s+) \leq Y^{\dagger}(s) \leq Y^{\sharp}(s)$ for $s > 0$,
it can be shown that ${\cal S}^-
\subseteq {\cal S} \subseteq 
{\cal S}^+$. 
Define $\tau = 
\sup {\cal S}^+;$ then, 
because $\sup {\cal S}^+ = \sup {\cal S}^-$,
it follows that 
$\tau = \sup {\cal S}$.

\begin{proof}[Proof of \cref{prop:sharp-martingale}]
    These results are well-known. For example, Theorems 1.3.1 and 1.3.2 of \citet{fleming1991counting} can be used to establish the result for $M_T^{\sharp};$ parallel calculations then prove that $M_C^{\sharp}$ is a ${\cal F}_t$-martingale. Uniqueness follows from Theorem 1.4.1 in \citet{fleming1991counting}. 
\end{proof}

\begin{proof}[Proof of \cref{prop:identif}]
For $t \in {\cal S}$, we see that
    \begin{equation*}
        \Lambda_C^{\dagger}(t)
        = \int_{(0, t]} \frac{\d \P(X \leq u, \Delta=0 )}{\P^*(T > u, C \geq u )} 
        = \int_{(0, t]} \frac{\P^*(T > u ) \,\d \P^*(C \leq u )}{\P^*(T > u, C \geq u )} 
        = \int_{(0, t]} \frac{\d \P^*(C \leq u )}{\P^*(C \geq u )}.
    \end{equation*}
    The second equality follows from iterated expectation, that is, 
    \begin{equation*}
        \P(X \leq t, \Delta=0 ) 
        =\, \P^*(C \leq t, T>C) 
        =\, \int_{(0, t]} \P^*(T > u ) \,\d \P^*(C \leq u).
    \end{equation*}
    Note the central importance of the asymmetric definition of $\Delta = I(T \leq C)$ in this expression.

    The proof that $\Lambda^\sharp_T(t) = \Lambda_T(t)$ for $t \in  {\cal S}^+$ is well-known and follows similarly. 
\end{proof}

\begin{proof}[Proof of \cref{cor:lambda-sharp-coerc}]
    \cref{sec:supp:basic-addntl:props} shows that $\Lambda_C^{\dagger}$ is right-continuous. To prove the first
    result, we consider the cases
    $t < \tau$ and 
    $t \geq \tau$ separately.

    First, suppose that $t < \tau$.
    Then, $\E \{Y^{\sharp}(t) \} \geq \E \{Y^{\dagger}(t) \} > 0$. 
    In this case, we may write
     \begin{align*}
        \Lambda_C^{\dagger}(t) - \Lambda_C^{\sharp}(t)
        & = \int_{(0,t]} \left\{ 1 -  \frac{\E\{Y^{\dagger}(u)\}}{\E\{Y^{\sharp}(u)\}} \right\} \frac{\d \E \{ N_C(u) \} }{\E\{Y^{\dagger}(u)\}} \\
        & = \int_{(0,t]} \frac{\E\{Y^{\sharp}(u)\} - \E\{Y^{\dagger}(u)\}}{\E\{Y^{\sharp}(u)\}} \frac{\d \E \{ N_C(u) \} }{\E\{Y^{\dagger}(u)\}} \\
        & = \int_{(0,t]} \frac{\delta \E\{N_T(u)\}}{\E\{Y^{\sharp}(u)\}} \frac{\d \E \{ N_C(u) \} }{\E\{Y^{\dagger}(u)\}} \\
        & = \sum_{ 0 < u \leq t} \frac{\delta \E\{ N_T(u) \} \, \delta \E \{ N_C(u) \}}{\E\{Y^{\sharp}(u)\} \E\{Y^{\dagger}(u)\}} \\
        & = \sum_{ 0 < u \leq t} 
        \delta \Lambda_T^\sharp(u)
        \delta \Lambda_C^\dagger(u).
    \end{align*}
    The summation in this last expression vanishes if and only if $\Lambda_T^\sharp(u)$ and 
    $\Lambda_C^\dagger(u)$ share no discontinuities on $(0,t]$. 
    This establishes for
    $t < \tau$ that $\Lambda_C^{\dagger}(t) = \Lambda_C^{\sharp}(t)$ if and only if $\Lambda_T^\sharp(u)$ and 
    $\Lambda_C^\dagger(u)$ share no discontinuities on $(0,t]$.

    Next, we consider the case where
    $t \geq \tau$. We begin by noting
    that for such $t$, we can write
    \[
\Lambda^\dagger_C(t) - \Lambda^\sharp_C(t)
 = 
 \left\{ \Lambda^\dagger_C(\tau-) - \Lambda^\sharp_C(\tau-) \right\}
 + 
 \left\{
 \delta \Lambda^\dagger_C( \tau ) - \delta \Lambda^\sharp_C( \tau ) \right\}
 +
 \left\{
\Lambda^\dagger_C(t) - \Lambda^\dagger_C(\tau) -
\Lambda^\sharp_C(t) + \Lambda^\sharp_C(\tau) 
\right\}.
\]
By the arguments given previously,
the first term is zero
if and only if 
$\delta \Lambda_T^\sharp(u)
\delta \Lambda_C^\dagger(u) = 0$
for every $u < \tau$.  
Hence, we must only establish
conditions under which the 
remaining terms on the right-hand side
are zero. To do so, it is helpful to
separately consider
the cases where $\tau \in {\cal S}^c$
and $\tau \in {\cal S}$.

Suppose $\tau \in {\cal S}^c$. Then,
note that $\E \{ Y^\sharp(u) \} = \P (X \geq u ) = 0$ for any $u \geq \tau$. 
Following arguments similar
to those in 
\cref{sec:supp:basic-addntl:props},
it is easy to show that 
$\P (vX \geq u ) = 0$ implies
that $\E \{ Y^\dagger(u) \} = 0$
and hence that $\delta \E\{ N_T(u) \}
= \delta \E \{N_C(u)\} = 0$. Therefore,
\[
\left\{
 \delta \Lambda^\dagger_C( \tau ) - \delta \Lambda^\sharp_C( \tau ) \right\}
 +
 \left\{
\Lambda^\dagger_C(t) - \Lambda^\dagger_C(\tau) -
\Lambda^\sharp_C(t) + \Lambda^\sharp_C(\tau) 
\right\} = 0
\]
there and the result stated 
in the proposition holds.

Now, suppose that $\tau \in {\cal S}$.
In this case, 
$(\tau,t] \in {\cal S}^c$ and the argument just given
can again be used to prove that
\[
    \Lambda^\dagger_C(t) - \Lambda^\dagger_C(\tau) -
\Lambda^\sharp_C(t) + \Lambda^\sharp_C(\tau) 
    = 0.
\]
It therefore suffices to consider
what happens to 
$
 \delta \Lambda^\dagger_C( \tau ) - \delta \Lambda^\sharp_C( \tau )
$
when $\tau \in {\cal S}$.
In this case,  we must have $\delta \Lambda^\dagger_C(\tau) > 0$
and $\delta \Lambda^\sharp_C(\tau) > 0;$
since these differ only by the choice
of denominator, equality holds if and
only if 
$\E\{ Y^\sharp(\tau) \} = 
\E\{ Y^\dagger(\tau) \}$.
However, 
$\E\{ Y^\sharp(\tau) \} = 
\E\{ Y^\dagger(\tau) \}$
if and only if 
$\delta \E\{ N_T(\tau) \} = 0$,
or equivalently, $\delta \Lambda^\dagger_T(\tau) = 0$. 
Hence, when
$\tau \in {\cal S}$,
$
 \delta \Lambda^\dagger_C( \tau ) - \delta \Lambda^\sharp_C( \tau ) = 0
$
if and only if 
$\delta \Lambda^\sharp_C( \tau )
\delta \Lambda^\sharp_T(\tau) = 0$.

Taken together, the above arguments
establish that, for each $t>0$,
$
\Lambda^\dagger_C( t ) = 
\Lambda^\sharp_C( t )
$
if and only if
\[
  \delta \Lambda_T^\sharp(u)
        \delta \Lambda_C^\dagger(u) = 0
\]
for all $u \in (0,t]$. %
\end{proof}

\begin{proof}[Proof of \cref{prop:dagger-martingale}]
    The stochastic process $M_C^{\dagger}(t)$ is clearly adapted. For $t>0$, $M_C^{\dagger}(t)$ is also absolutely integrable since 
    \begin{align*}
     \E \left\{ \left|N_C(t) - \int_{(0,t]} Y^{\dagger}(u) \,\d \Lambda_C^{\dagger}(u) \right| \right\}
     & \leq \E \{ N_C(t) \} +  \E \left\{ \int_{(0,t]} Y^{\dagger}(u) \,\d \Lambda_C^{\dagger}(u) \right\}  \\
     & \hspace{-25mm} = \E \{ N_C(t) \} + \int_{(0,t]} \E \{ Y^{\dagger}(u) \} \,\d \Lambda_C^{\dagger}(u) \\
     & \hspace{-25mm} = \E \{ N_C(t) \} + \int_{(0,t]} \E \{ Y^{\dagger}(u) \} \frac{\d \E \{N_C(u)\} }{\E\{Y^{\dagger}(u)\}} \\
     & \hspace{-25mm} = \E \{ N_C(t) + N_C(t)\} \\
     & \hspace{-25mm} \leq 2,
    \end{align*}
    where the interchange of expectations and integration in the first equality is justified by Tonelli's theorem. 
    
    Finally, we establish the no-drift property that $\E\{ M(t+s) \mid \F_t \} = M(t)$ for any $s>0$. After rewriting these expressions in the same manner as \citet{fleming1991counting}, we find that we must show that
    \begin{equation}
    \label{mart-condition}
     \E\{ N_C(t+s) - N_C(t) \mid \F_t \} 
     = \E \left\{ \int_{(t,t+s]} Y^{\dagger}(u) \,\d \Lambda_C^{\dagger}(u) \middle| \F_t \right\}.
    \end{equation}
    When $Y^{\sharp}(t+)=0$ almost surely, this follows immediately since both expressions vanish.
    
    Now assume $\E\{Y^{\sharp}(t+)\}>0$; the first term simplifies as 
    \begin{align*}
        \E\{ N_C(t+s) - N_C(t) \mid \F_t \} 
        & = \E \{ N_C(t+s) - N_C(t) \mid N_T(u), N_C(u),\, 0 \leq u \leq t \} \\
        & \hspace{-15mm} = Y^{\sharp}(t+) \E \{ N_C(t+s) - N_C(t) \mid Y^{\sharp}(t+)=1 \} \\
        & \hspace{-15mm} = Y^{\sharp}(t+) \frac{\E \{ N_C(t+s) - N_C(t) \}}{\E \{Y^{\sharp}(t+)\} }.
    \end{align*}
    Similarly, the second term simplifies as
    \begin{align}
        \E \left\{ \int_{(t,t+s]} Y^{\dagger}(u) \,\d \Lambda_C^{\dagger}(u) \middle| \F_t \right\} 
        =\, & Y^{\sharp}(t+) \E \left\{ \int_{(t,t+s]} Y^{\dagger}(u) \,\d \Lambda_C^{\dagger}(u) \middle| Y^{\sharp}(t+)=1 \right\} \nonumber \\
        & \hspace{-25mm} = Y^{\sharp}(t+) \int_{(t,t+s]} \E \left\{ Y^{\dagger}(u) \middle| Y^{\sharp}(t+)=1 \right\} \,\d \Lambda_C^{\dagger}(u) \nonumber \\
        & \hspace{-25mm} = Y^{\sharp}(t+) \int_{(t,t+s]} \frac{\E \{ Y^{\dagger}(u)\}}{\E \{ Y^{\sharp}(t+) \}} \,\d \Lambda_C^{\dagger}(u) \label{eq:mart-arg} \\
        & \hspace{-25mm} = \frac{Y^{\sharp}(t+)}{\E \{ Y^{\sharp}(t+) \}} \int_{(t,t+s]} \E \{ Y^{\dagger}(u)\} \,\d \Lambda_C^{\dagger}(u) \nonumber \\
        & \hspace{-25mm} = \frac{Y^{\sharp}(t+)}{\E \{ Y^{\sharp}(t+) \}} \int_{(t,t+s]} \,\d \E \{N_C(u)\}. \nonumber
    \end{align}
    The desired result \eqref{mart-condition} now follows since $\int_{(t,t+s]} \d \E \{N_C(u)\} = \E \{ N_C(t+s) - N_C(t) \}$.
\end{proof}

\begin{proof}[Proof of \cref{lem:mcd-properties}]
    The centering term may be expressed as 
    \begin{equation}
        \int_{(0,t]} Y^{\dagger}(u) \,\d \Lambda_C^{\dagger}(u)
        = \begin{cases}
              I(X>t) \,\Lambda_C^{\dagger}(t)+ I(X \leq t) \Lambda_C^{\dagger}(X) & \text{ when } \Delta = 0, \\
              I(X>t) \,\Lambda_C^{\dagger}(t) + I(X \leq t) \Lambda_C^{\dagger}(X-) & \text{ when } \Delta = 1. 
          \end{cases} \label{eq:centterm-rc}
    \end{equation}
    \cref{sec:supp:basic-addntl:props} shows that $\Lambda_C^{\dagger}$ is right-continuous; hence, in each case, the centering term is right-continuous. Since $N_C$ is right-continuous, it follows that
    $M_C^{\dagger}$ is also right-continuous. 
    
   We now show that the centering term is not $\F_t$-predictable by considering the decomposition
    \begin{align*}
        \int_{(0,t]} Y^{\dagger}(u) \,\d \Lambda_C^{\dagger}(u)
        & = \int_{(0,t]} Y^{\sharp}(u) \,\d \Lambda_C^{\dagger}(u) 
        - \int_{(0,t]} \delta N_T(u) \,\d \Lambda_C^{\dagger}(u) \\
        & \hspace{-15mm} = \int_{(0,t]} Y^{\sharp}(u) \,\d \Lambda_C^{\dagger}(u) 
        - \int_{(0,t)} \delta N_T(u) \,\d \Lambda_C^{\dagger}(u) 
        - \delta N_T(t) \delta \Lambda_C^{\dagger}(t) \\
        & \hspace{-15mm} = \int_{(0,t]} Y^{\sharp}(u) \,\d \Lambda_C^{\dagger}(u) 
        - I(\Delta=1, X<t) \left\{ \Lambda_C^{\dagger}(X) - \Lambda_C^{\dagger}(X-) \right\}
        - \delta N_T(t) \delta \Lambda_C^{\dagger}(t), 
    \end{align*}
    valid for all $t>0$. 
    The first summand is $\F_t$-predictable \citep[, e.g. Proposition 1.4.2]{fleming1991counting}. The second summand is also $\F_t$-predictable as it is adapted and left-continuous \citep[Lemma 1.4.1]{fleming1991counting}.
    Since the first and second summands are $\F_t$-predictable , the centering term is $\F_t$-predictable if and only if the final summand is $\F_t$-predictable. However, $\delta N_T(t) \delta \Lambda_C^{\dagger}(t)$ is clearly not $\F_t$-predictable since, for $t>0$, $\delta N_T(t) = N_T(t) - N_T(t-)$ depends on $N_T(t)$ which is not in $\F_{t-}$ \citep[Proposition 1.4.1, Exercise 1.6]{fleming1991counting}. 
\end{proof}

Before the proof of \cref{cor:marting-sharp-coerc}, we provide two lemmas. 

\begin{lem}
\label{lem:mod.eq.indist}
Let $U(t)$ and $V(t)$ be two right-continuous processes for $t \geq 0.$
Then, $U$ and $V$ are 
indistinguishable 
if and only if $U(t) = V(t)$ almost surely for each $t \geq 0.$
\end{lem}

\begin{proof}
The result is well-known in the literature. 
If $U$ and $V$ are 
indistinguishable,
$U(t) = V(t)$ are almost surely for each $t \geq 0.$ Now, suppose only
that 
$U(t) = V(t)$ almost surely for each $t \geq 0.$ Then,
by countable additivity, the set ${A=\{\omega\colon V(t,\omega)=U(t,\omega){\rm\ for\ all\ }t\in{\mathbb Q}_+\}}$ has probability one. By right-continuity of the sample paths, it follows that $U(t)$ and $V(t)$ necessarily have the same paths for all ${\omega\in A}.$
\end{proof}

\begin{lem}
\label{lem:Znon}
Let $Z$ be a nonnegative
random variable. Then,
$\P(Z = 0) = 1$ if and
only if $\E(Z) = 0.$
\end{lem}

\begin{proof}[Proof of \cref{cor:marting-sharp-coerc}]
    Recall that the stochastic processes $M_C^{\dagger}$ and $M_C^{\sharp}$ are indistinguishable if their difference vanishes uniformly in $t$,  that is, if 
    \begin{equation}
        \P \left\{ M_C^{\dagger}(t) - M_C^{\sharp}(t) = 0 \text{ for all } t>0 \right\}
        = 1. \label{eq:indist}
    \end{equation}

We begin by noting that
   \begin{align}
        & M_C^{\dagger}(t) - M_C^{\sharp}(t) 
        = \left\{ N_C(t) - \int_{(0,t]} Y^{\dagger}(u) \,\d \Lambda_C^{\dagger}(u) \right\} - \left\{ N_C(t) - \int_{(0,t]} Y^{\sharp}(u) \,\d \Lambda_C^{\sharp}(u) \right\} \nonumber \\
        & \hspace{15mm} = \int_{(0,t]} Y^{\sharp}(u) \,\d \Lambda_C^{\sharp}(u) - \int_{(0,t]} Y^{\dagger}(u) \,\d \Lambda_C^{\dagger}(u) \nonumber \\
        & \hspace{15mm} = \int_{(0,t]} Y^{\sharp}(u) \,\d \Lambda_C^{\sharp}(u) - \int_{(0,t]} Y^{\sharp}(u) \,\d \Lambda_C^{\dagger}(u) + \int_{(0,t]} \delta N_T(u) \,\d \Lambda_C^{\dagger}(u) \nonumber \\
        & \hspace{15mm} = \int_{(0,t]} Y^{\sharp}(u) \,\d \left\{ \Lambda_C^{\sharp}(u) -  \Lambda_C^{\dagger}(u) \right\} + \int_{(0,t]} \delta \Lambda_C^{\dagger}(u) \,\d  N_T(u). 
        \label{eq:m-diff-pre-new}
\end{align} 

In particular,
since the expression \eqref{eq:m-diff-pre-new}
is right-continuous
in $t,$ the proof follows
immediately from \cref{lem:mod.eq.indist}
provided that, for each $t>0$, the expression \eqref{eq:m-diff-pre-new} is identically
zero if and only
$\delta \E\{N_T(u)\}
    \delta \E\{N_C(u)\}=0$ for all $u \in (0,t]$.
By 
\cref{cor:lambda-sharp-coerc}, 
the first term on the
right-side is
zero for each $t$ if and only if $\delta \E\{N_T(u)\}
    \delta \E\{N_C(u)\}=0$ for all $u \in (0,t]$.
Since that the first term on the
right-side is nonnegative, it suffices
to focus on
the behavior of
the right-continuous,
nonnegative process
\[
Z(t) = \int_{(0,t]} \delta \Lambda_C^{\dagger}(u) \,\d  N_T(u) =
\sum_{0 < u \leq t} \frac{\delta \E \{ N_C(u) \}}{\E \{ Y^{\dagger}(u) \}} \,\delta N_T(u).
\]

Fix an arbitrary $t>0.$
By \cref{lem:Znon}, it suffices to
show that $\E\{Z(t)\} = 0$
if and only if
$\delta \E\{N_T(u)\}
    \delta \E\{N_C(u)\}=0$ for all $u \in (0,t]$.
However, since
\[
\E\{Z(t) \} = 
\sum_{0 < u \leq t} \frac{\delta \E \{ N_C(u) \}}{\E \{ Y^{\dagger}(u) \}} \,\delta \E \{N_T(u) \},
\]
arguments similar
to those in \cref{cor:lambda-sharp-coerc} show that
$\E\{Z(t) \} = 0$
if and only if
$\delta \E\{N_T(u)\}
    \delta \E\{N_C(u)\}=0$ for all $u \in (0,t]$,
or equivalently if and only if $\delta \Lambda_T^{\sharp}(u)
    \delta \Lambda_C^{\dagger}(u) = 0$ for all $u \in (0,t]$. 
\end{proof}

\section{Additional Material for \texorpdfstring{\cref{sec:martingale-trans}}{Section 3}}
\label{sec:supp:mt-addntl}

\begin{proof}[Proof of \cref{lem:half-pred}]
    Denote $\mathcal{P}$ as the $\F_t$-predictable sigma algebra. Since $\F_{t-} \subseteq \G_t$, the set of half-predictable rectangles contains the set of rectangles
    \begin{equation*}
        \{0\} \times A, \, A \in \F_0 
        \hspace{2mm} \text{ and } \hspace{2mm} 
        [a,b) \times A, \, 0 < a \leq b \leq \infty, A \in \F_{a-},
    \end{equation*}
    which generates $\mathcal{P}$ \citep[Cor. 6.3.3]{elliott1982stochastic}. Therefore the $\mathcal{P} \subseteq \H$, proving the first claim. 

    We now establish the second claim. Let $t>0$. For $B \in \H$, define $B_t = \{\omega : (t, \omega) \in B \}$. Using that $\G_t$ is a filtration, we may readily see that $\{ B \in \H \, : \, B_t \in \G_{t} \} = \H$. Thus for any set $B \in \H$, we have that $B_t \in \G_{t}$ so that $I_{B}(t, \cdot) = I_{B_t}$ is $\G_{t}$-measurable as its pullback is either $B_t$ or its complement. The half-predictable process $H$ may be expressed as
    \begin{equation*}
        H(t, \omega)
        = \lim_{m \to \infty} \sum_{j=1}^m k_{j,m} I_{B_{j,m}} (t, \omega). 
    \end{equation*}
    for a triangular array of constants $k_{j,m}$ and $B_{j,m} \in \H$. For any fixed $m$, the summation $\sum_{j=1}^m k_{j,m} I_{B_{j,m}} (t, \cdot)$ is $\G_t$-measurable as each summand is $\G_t$-measurable. Thus the pointwise limit $H(t, \cdot)$ is also $\G_t$-measurable. 
\end{proof}

\subsection{Proof of martingale transform}
\label{sec:supp:mt-addntl:proofs}

\begin{lem}
\label{lem:elem-mart}
    If $B$ is a half-predictable rectangle, then the process 
    \begin{equation*}
        \int_{(0,\cdot]} I_B(u) \,\d M_C^{\dagger}(u)
    \end{equation*}
    is an $\F_t$-martingale. 
\end{lem}

\begin{proof}
    Define $L(t) := \int_{(0,t]} I_B(u) \,\d M_C^{\dagger}(u)$
    For $B=\{0\} \times A$ with $A \in \G_0$, $I_B(u)=0$ for all $u>0$ so $L(t) = 0$ is obviously an $\F_t$-martingale.

    Let $B = [a,b) \times A$ for $0 < a < b \leq \infty, A \in \G_a$. The process $L$ is clearly $\F_t$-adapted. Similarly, 
    \begin{equation*}
        \E \{ | L(t) | \}
        \leq \E \left\{ \int_{(0,t]} |\d M_C^{\dagger}(u)| \right\}
        \leq 2 
        < \infty. 
    \end{equation*}
    Finally, for $t,s>0$, the process $L$ has no drift since 
    \begin{align*}
        \E \left\{ \int_{(t,t+s]} I_B(u) \,\d M_C^{\dagger}(u) \middle| \F_t \right\}
        & = \E \left\{ I_A \int_{(t,t+s] \cap [a,b)} \,\d M_C^{\dagger}(u) \middle| \F_t \right\} \nonumber \\
        & \hspace{-30mm} = \E \left\{ I_A \int_{(t,t+s] \cap (a,b]} \,\d M_C^{\dagger}(u) \middle| \F_t \right\} 
         + \E \left\{ I_A I(t < a \leq t+s) \delta M_C^{\dagger}(a) \middle| \F_t \right\} \nonumber 
         - \E \left\{ I_A I(t < b \leq t+s) \delta M_C^{\dagger}(b) \middle| \F_t \right\} \nonumber \\
        & \hspace{-30mm} = \E \left\{ I_A \int_{(a \vee t, b \wedge (t+s)]} \,\d M_C^{\dagger}(u) \middle| \F_t \right\} \\
        & \hspace{-30mm} = \E \left[ I_A \E \left\{ \int_{(a \vee t, b \wedge (t+s)]} \,\d M_C^{\dagger}(u) \middle| \F_{a \vee t} \right\} \middle| \F_t \right] \\
        & \hspace{-30mm} = 0. 
    \end{align*}
    vanishes, where $I_A$ denotes the indicator of the set $A$. 
\end{proof}

We now prove the main result. 

\begin{proof}[Proof of \cref{thm:martingale-transform}]
    We proceed using the same proof strategy as Theorem 1.5.1 in \citet{fleming1991counting}. %

    Define $\mathcal{V}$ as the vector space of bounded and adapted $H$ such that $\int_{(0,\cdot]} H(u) \,\d M_C^{\dagger}(u)$ is an $\F_t$-martingale. Let $H_j \in \mathcal{V}$ for $j=1,\dots$ to be an increasing sequence such that $H := \sup_j H_n$ is almost surely bounded. We will show that $H \cdot M_C^{\dagger}$ is an $\F_t$-martingale; it clearly is absolutely integrable and adapted. For $t,s>0$, it is drift-free since
    \begin{align*}
        \E \left\{ \int_{(0,t+s]} H(u) \,\d M_C^{\dagger}(u) \middle| \F_t \right\}
        & = \E \left\{ \int_{(0,t+s]} \lim_{j\to\infty} H_j(u) \,\d M_C^{\dagger}(u) \middle| \F_t \right\} \\
        & \hspace{-10mm} = \lim_{j\to\infty} \E \left\{ \int_{(0,t+s]} H_j(u) \,\d M_C^{\dagger}(u) \middle| \F_t \right\} \\
        & \hspace{-10mm} = \lim_{j\to\infty} \E \left\{ \int_{(t,t+s]} H_j(u) \,\d M_C^{\dagger}(u) \middle| \F_t \right\} \\
        & \hspace{-10mm} = \E \left\{ \int_{(t,t+s]} H(u) \,\d M_C^{\dagger}(u) \middle| \F_t \right\}.
    \end{align*}
    Therefore $H \cdot M_C^{\dagger}$ is an $\F_t$-martingale and so $\mathcal{V}$ is closed under supremums over the increasing sequence $H_j$, $j=1, \dots$. 

    Since $\mathcal{V}$ also contains the constant functions and, by \cref{lem:elem-mart}, contains the indicators of the half-predictable rectangles, the result now follows from Lemma 1.5.2 in \citet{fleming1991counting}, a corollary of the monotone class theorem. 
\end{proof}

\subsection{Proof of martingale transform in discrete time}
\label{sec:supp:mt-addntl:proofs-discrete}

We now give a simplified statement and proof of \cref{thm:martingale-transform} that is valid in discrete time and avoids measure theoretic technicalities. 

\begin{thm}
\label{thm:martingale-transform-discrete}
 Assume $\E \{ N_C(\cdot) \}$ is a step-function. If $H$ is $\G_t$-adapted and bounded, then $H \boldsymbol{\cdot} M^{\dagger}_{C}$ is a $\F_t$-martingale.
\end{thm}

\begin{proof}
    Assume that censoring is discrete so that the identified cumulative hazard for censoring is
    \begin{equation*}
        \Lambda_C^{\dagger}(t)
        = \sum_{0 < u \leq t} \delta \Lambda_C^{\dagger}(u). 
    \end{equation*}
    Let $D$ be the set of discontinuities. Then $N_C$ jumps on $D$ almost surely. 
    
    Since $\G_t \subseteq \F_t$, the transform $H \boldsymbol{\cdot} M^{\dagger}_{C} = \int_{(0,\cdot]} H(u) \,\d M_C^{\dagger}(u)$
    is clearly adapted. 
    Let $t>0$ and suppose that $H$ is bounded by $C<\infty$. The integrability follows from the usual argument that
    \begin{align*}
        & \E \left\{ \left| \int_{(0,t]} H(u) \,\d M_C^{\dagger}(u) \right| \right\}
        = \E \left\{ \left| \sum_{(0,t] \cap D} H(u) \,\delta M_C^{\dagger}(u) \right| \right\} \\
        & \hspace{15mm} \leq \sum_{(0,t] \cap D} \E \left\{ \left| H(u) \,\delta M_C^{\dagger}(u) \right| \right\} \\
        & \hspace{15mm} \leq C \sum_{(0,t] \cap D} \biggl[ \E \left\{ \,\delta N_C(u) \right\} + \E \left\{ Y^{\dagger}(u) \frac{\delta \E \{ N_C(u)\} }{\E \{Y^{\dagger}(u) \}} \right\} \biggr] \\
        & \hspace{15mm} \leq 2C.
    \end{align*}

    Let $s>0$. 
    The conditional expectation may be evaluated as
    \begin{align*}
        & \E \left\{ \int_{(t,t+s]} H(u) \, \d M_C^{\dagger}(u) \middle| \F_t \right\} 
        = \E \left[ \sum_{(t,t+s] \cap D} H(u) \,\delta M_C^{\dagger}(u) \middle| \F_t \right] \\
        & \hspace{15mm} = \E \left[ \sum_{(t,t+s] \cap D} \E \biggl\{ H(u) \,\delta M_C^{\dagger}(u) \bigg| \G_u \biggr\} \middle| \F_t \right] \\
        & \hspace{15mm} = \E \left[ \sum_{(t,t+s] \cap D} H(u) \, \E \left\{ \delta M_C^{\dagger}(u) \middle| \G_u \right\} \middle| \F_t \right] \\
        & \hspace{15mm} = 0,
    \end{align*}
    The first equality uses that $\F_t \subseteq \G_u$ for $u > t$, the second uses the half-predictability of $H$. The third identity uses that
    \begin{align*}
        \E \left\{ \delta N_C(u) \middle| \G_u \right\}
        = Y^{\dagger}(u) \E \left\{ \delta N_C(u) \middle| Y^{\dagger}(u) = 1 \right\} 
        = Y^{\dagger}(u) \frac{\E \left\{ \delta N_C(u) \right\}}{\E \{ Y^{\dagger}(u) \}} 
        = Y^{\dagger}(u) \delta \Lambda_C^{\dagger}(u). 
    \end{align*}
    Therefore the conditional expectation vanishes and hence the transform $H \boldsymbol{\cdot} M^{\dagger}_{C}$ is drift-free. 

    Since the three conditions defining a martingale are verified, we may conclude that the transform is a martingale. %
\end{proof}

\subsection{Proof of variation and covariation processes}
\label{sec:supp:mt-addntl:var-and-covar}

We start with a series of technical lemmas. 

\begin{lem}
\label{lem:stp0}
    $\int_{(0,\cdot]} M_C^{\dagger}(u-) \,\d M_C^{\dagger}(u)$ is a local $\F_t$-martingale. 
\end{lem}

\begin{proof}
    The process is a local $\F_t$-martingale by Corollary 2.4.1 in \citet{fleming1991counting} since no special properties of $A_C^{\dagger}$ are used in the proof; see also Lemma 2.4.1 and Theorem 2.4.1 in the same book. 
\end{proof}

\begin{lem}
\label{lem:stp1}
    The process 
    \begin{equation*}
        \left( M_C^{\dagger} \right)^2 - \int_{(0,\cdot]} \frac{G^{\dagger}(u)}{G^{\dagger}(u-)} Y^{\dagger}(u) \,\d \Lambda_C^{\dagger}(u)
    \end{equation*}
    is a local $\F_t$-martingale. 
\end{lem}

\begin{proof}
    The algebraic identity
    \begin{equation}
        \left( M_C^{\dagger}(t) \right)^2 - \int_{(0,t]} \{ 1 - \delta A_C^{\dagger}(u) \} \,\d A_C^{\dagger}(u)
        = 2 \int_{(0,t]} M_C^{\dagger}(u-) \,\d M_C^{\dagger}(u) + M_C^{\dagger}(t) - 2 \int_{(0,t]} \delta A_C^{\dagger}(u) \,\d M_C^{\dagger}(u) \label{eq:algm}
    \end{equation}
    holds for all $t>0$, as shown in the proof of Theorem 2.6.1 in \citet{fleming1991counting}. The middle and last summand on the RHS of \cref{eq:algm} are both an $\F_t$-martingale by \cref{thm:martingale-transform}; the first summand on the RHS of \cref{eq:algm} is a local $\F_t$-martingale by \cref{lem:stp0}. 
\end{proof}

Let $\mathrm{Exp}(1)$ denote a random variable with an exponential distribution with unit rate. 

\begin{lem}
\label{lem:exp}
    Let $X>0$ be a random variable cumulative hazard $\Lambda$. %
    Then $\Lambda(X)$ is stochastically dominated by $\mathrm{Exp}(1)$, i.e. $\P \{ \mathrm{Exp}(1) \leq u \} \leq \P \{ \Lambda(X) \leq u \}$ for all $u>0$. 
\end{lem}

\begin{proof}
    Define $A = \Lambda (X)$ and $B \sim \lfloor \mathrm{Exp}(1) \rfloor_{\mathcal{S}}$; our goal is to show that $\P(B \leq u) < \P(A \leq u)$ for all $u>0$. 
    
    First, consider $a_0 \in \mathcal{S}$; let $x_0 = \inf \{ x \, : \, \Lambda(x) = a_0 \}$. By right-continuity of $\Lambda$, we know $\Lambda(x_0) = a_0$. 
    Also, 
    \begin{align*}
        \P (A \leq a_0)
        & = \P \{ x \, : \, \Lambda(x) \leq \Lambda(x_0) \} \\
        & \hspace{-10mm} = \P (x \, : x \leq x_0) + \P \{ x \, : \, x > x_0, \Lambda(x) = \Lambda(x_0) \} \\
        & \hspace{-10mm} \geq \P (x \, : x \leq x_0) \\
        & \hspace{-10mm} = 1 - \Prodi_{(0,x_0]} \left\{ 1 - \d \Lambda(x) \right\}
    \end{align*}
    since $\Lambda$ is monotonically non-decreasing. The remaining term involves the product integral which satisfies 
    \begin{equation*}
        \Prodi_{(0,x_0]} \left\{ 1 - \d \Lambda(x) \right\} 
        = \exp \left\{ - \Lambda(x_0) \right\} \prod_{0 < x \leq x_0} \left[ \left\{ 1 + \delta \Lambda(x) \right\} \exp \left\{ - \delta \Lambda(x) \right\} \right]
        \leq \exp \left\{ - \Lambda(x_0) \right\},
    \end{equation*}
    where the equality follows from Definition 4 in \citet{gill1990survey} and the inequality follows from $(1+\lambda)e^{-\lambda} \leq 1$ for all $\lambda \in [0,1]$. 
    Since $\Lambda(x_0) = a_0$ by construction, we have shown that $\P \{ \mathrm{Exp}(1) \leq a_0 \} \leq \P \{ \Lambda(X) \leq a_0 \}$ for all $a_0 \in \mathcal{S}$. 

    Now, consider $a_0 \not\in \mathcal{S}$. Put $a_{+} = \inf \{ a \, : \, a \geq a_0, a \in \mathcal{S} \}$. 
    Since $\Lambda$ is right continuous, $a_{+} \in \mathcal{S}$. Thus the previous argument, applied to ``$<$'' rather than ``$\leq$'', shows that $\P ( A < a_{+} ) \geq 1-\exp \left\{ - a_{+} \right\}$. The result follows since $\P ( A \leq a_0 ) = \P ( A < a_{+} )$ by the monotonicity of $\Lambda$. 
\end{proof}

\begin{lem}
\label{lem:ll-mart2}
    $\int_{(0,\cdot]} M_C^{\dagger}(u-) \,\d M_C^{\dagger}(u)$ is an $\F_t$-martingale. 
\end{lem}

\begin{proof}
    Define $L(t) := \int_{(0,\cdot]} M_C^{\dagger}(u-) \,\d M_C^{\dagger}(u)$. 
    The process $L$ is a $\F_t$-martingale if 
    \begin{align}
        \E \int_{(0,\infty)} \left( M_C^{\dagger}(u-) \right)^2 \,\d \langle M_C^{\dagger}, M_C^{\dagger} \rangle (u)
        & = \E \int_{(0,\infty)} \left( M_C^{\dagger}(u-) \right)^2 \frac{G^{\dagger}(u)}{G^{\dagger}(u-)} Y^{\dagger}(u) \,\d \Lambda_C^{\dagger}(u) \nonumber \\
        & \hspace{-20mm} \leq \E \int_{(0,\infty)} \left\{ M_C^{\dagger}(u-) \right\}^2 Y^{\dagger}(u) \,\d \Lambda_C^{\dagger}(u) \label{eq:llm2-eq1}
    \end{align}
    is finite, where the equality concerning the local variation process is from \cref{lem:stp1} \citep[Theorem 2.4.4]{fleming1991counting}. 

    The integrand in \cref{eq:llm2-eq1} may be written 
    \begin{align*}
        \left\{ M_C^{\dagger}(u-) \right\}^2
        & = \left\{ N_C(u-) - \int_{(0,t)} Y^{\dagger}(u) \,\d \Lambda_C^{\dagger}(u) \right\}^2 \\
        & \hspace{-10mm} = N_C(u-) -2 N_C(u-) \int_{(0,u)} Y^{\dagger}(v) \,\d \Lambda_C^{\dagger}(v) + \left\{ \int_{(0,u)} Y^{\dagger}(v) \,\d \Lambda_C^{\dagger}(v) \right\}^2. 
    \end{align*}
    Since $N_C(u-) Y^{\dagger}(u) = 0$ for all $u>0$, we see that \cref{eq:llm2-eq1} satisfies 
    \begin{align}
        \E \int_{(0,\infty)} \left\{ M_C^{\dagger}(u-) \right\}^2 \,\d \langle M_C^{\dagger}, M_C^{\dagger} \rangle (u)
        & \leq \E \int_{(0,\infty)} \left\{ \int_{(0,u)} Y^{\dagger}(v) \,\d \Lambda_C^{\dagger}(v) \right\}^2 Y^{\dagger}(u) \,\d \Lambda_C^{\dagger}(u) \nonumber \\
        & \leq \E \left[ \left\{ \int_{(0,\infty)} Y^{\dagger}(u) \,\d \Lambda_C^{\dagger}(u) \right\}^3 \right] \tag{since the squared term is increasing in $u$} \nonumber \\
        & \leq \E \left[ \left\{ \Lambda_C^{\dagger}(X) \right\}^3 \right]. \label{eq:third-moment}
    \end{align}
    If $T \ind C$, then \cref{lem:exp} shows that 
    \begin{align*}
        \E \{ \Lambda_C^{\dagger}(X)^3 \}
        = \E \{ \Lambda_C(X)^3 \}
        \leq \E \{ \Lambda_C(C)^3 \}
        \leq \E \left\{ \mathrm{Exp}(1)^3 \right\}
        = 6
        < \infty,
    \end{align*}
    where the first equality follows from \cref{prop:identif} as the expectation is only over $u$ such that $\E \{ Y^{\dagger}(u) \} > 0$. When $T \not\ind C$, there exists latent variables such that the same arguments holds with $\Lambda_C^{\dagger} \leq \Lambda_C$ \citep[see e.g.][]{peterson1976bounds}. 
\end{proof}

\begin{prop}
    \label{prop:dag-square-mart}
    The process 
    \begin{equation*}
        \left( M_C^{\dagger} \right)^2 - \int_{(0,\cdot]} \frac{G^{\dagger}(u)}{G^{\dagger}(u-)} Y^{\dagger}(u) \,\d \Lambda_C^{\dagger}(u)
    \end{equation*}
    is an $\F_t$-martingale. 
\end{prop}

\begin{proof}
    The follows immediately from the algebraic decomposition \cref{eq:algm} in the proof of \cref{lem:stp1} and from \cref{lem:ll-mart2}. 
\end{proof}

We now prove the main result of this section. 

\begin{proof}[Proof of \cref{thm:covar-proc}]
    Define the stochastic process $Q$ so that 
    \begin{equation*}
        Q(t)
        := (H_1 \cdot M_C^{\dagger})(t) (H_2 \cdot M_C^{\dagger})(t) - \int_{(0,t]} H_1(u) H_2(u) \frac{G^{\dagger}(u)}{G^{\dagger}(u-)} Y^{\dagger}(u) \,\d \Lambda_C^{\dagger}(u).  
    \end{equation*}
    The proof that $Q$ is an $\F_t$-martingale will have three steps: first we establish the claim when $H_1$ and $H_2$ are simple, second we establish the claim when $H_1$ is simple and $H_2$ is bounded, and third we establish the full claim. 

    First, suppose that $H_1$ and $H_2$ are simple half-predictable processes so that $H_j = I_{B_j}$ for a half-predicable $B_j$, defined in \cref{eq:half-pred-rect}. We must show that $Q$ is adapted, absolutely integrable, and is drift-free. $Q$ is clearly adapted, and
    \begin{align*}
        & \E \left\{ \left| (H_1 \cdot M_C^{\dagger})(t) (H_2 \cdot M_C^{\dagger})(t) - \int_{(0,t]} H_1(u) H_2(u) \frac{G^{\dagger}(u)}{G^{\dagger}(u-)} Y^{\dagger}(u) \,\d \Lambda_C^{\dagger}(u) \right| \right\} \\
        & \hspace{15mm} \leq \E \left\{ (M_C^{\dagger})^2(t) \right\} + \E \left\{ \int_{(0,t]} Y^{\dagger}(u) \,\d \Lambda_C^{\dagger}(u) \right\} \\
        & \hspace{15mm} \leq 2 \, \E \left\{ \int_{(0,t]} Y^{\dagger}(u) \,\d \Lambda_C^{\dagger}(u) \right\} \\
        & \hspace{15mm} = 2 \, \E \left\{ N_C(u) \right\}
        < \infty, 
    \end{align*}
    where the second inequality follow from \cref{prop:dag-square-mart} and the third equality follows from \cref{prop:dagger-martingale}. 
    As the final matter in this first case, we now show that $Q$ is drift-free. This follows from the same argument in \citet[Theorem 2.4.2]{fleming1991counting}. Thus $Q$ is an $\F_t$-martingale when $H_1, H_2$ are simple. 

    Second, it follows that $Q$ is an $\F_t$-martingale when $H_1$ is simple through the same monotone class argument in \citet[Theorem 2.4.2]{fleming1991counting}.

    Third, it follows that $Q$ is an $\F_t$-martingale when $H_2$ is bounded through another monotone class argument, as in \citet[Theorem 2.4.2]{fleming1991counting}. This is the final case, so the proof that $Q$ is an $\F_t$-martingale is complete. The proof for the stochastic process $(H_1 \cdot M_C^{\dagger})(H_3 \cdot M_T^{\sharp})$ is similar so omitted. 
\end{proof}

\subsection{Proof of variation and covariation processes in discrete time}
\label{sec:supp:mt-addntl:var-and-covar-discrete}

In this section we state and prove a version of \cref{thm:covar-proc} that is valid in discrete time and avoids measure theoretic technicalities. For simplicity, we make use of the lemmas in \cref{sec:supp:mt-addntl:var-and-covar}. 

\begin{thm}
\label{thm:covar-proc-discrete}
 Assume $X$ is discrete. If $H_1$ and $H_2$ are $\G_t$-adapted and bounded, then the following is a $\F_t$-martingale 
 \begin{equation}
    (H_1 \boldsymbol{\cdot} M^{\dagger}_{C}) (t) \, (H_2 \boldsymbol{\cdot} M^{\dagger}_{C}) (t) - \int_{(0,t]} H_1(u) H_2(u) \frac{G^{\dagger}(u)}{G^{\dagger}(u-)} Y^{\dagger}(u) \,\d \Lambda_C^{\dagger}(u). \label{eq:half-pred-var}
 \end{equation}
 If $H_1$ is $\G_t$-adapted and bounded and $H_3$ is $\F_t$-predictable and bounded, then the following is a $\F_t$-martingale 
 \begin{equation*}
    (H_1 \boldsymbol{\cdot} M^{\dagger}_{C}) (t) \, (H_3 \boldsymbol{\cdot} M^{\sharp}_{T}) (t). 
 \end{equation*}
\end{thm}
 
\begin{proof}
    We start by considering the variation process. %
    Adaptedness and integrability are straightforward to verify, so we focus on proving the no-drift property. 

    Let $s,t>0$. By expanding the integrals over disjoint intervals $(0,t]$ and $(t, t+s]$, we can show that 
    \begin{align}
        & \E \left\{ \int_{(0,t+s]} H_1(u) \,\d M_C^{\dagger}(u) \int_{(0,t+s]} H_2(u) \,\d M_C^{\dagger}(u) \,\middle|\, \F_t \right\} \nonumber \\
        =\, & \int_{(0,t]} H_1(u) \,\d M_C^{\dagger}(u) \int_{(0,t]} H_2(u) \,\d M_C^{\dagger}(u) 
        + \E \left\{ \int_{(t,t+s]} H_1(u) \,\d M_C^{\dagger}(u) \int_{(t,t+s]} H_2(u) \,\d M_C^{\dagger}(u) \,\middle|\, \F_t \right\}. \label{eq:covar-cent-basic}
    \end{align}
    We may determine the desired centering term by expressing the conditional expectation in \cref{eq:covar-cent-basic} as a contrast between the centering term evaluated at $t+s$ and at $t$. Note, if the centering term is an integral over $(0, \cdot]$, then this contrast will simply be an integral over $(t, t+s]$. 
    
    By expanding the interval $(t,t+s] = (t,v) \cup \{v\} \cup (v, t+s]$, the double integral in the conditional expectation may be written as
    \begin{align}
        \int_{(t,t+s]} H_1(u) \,\d M_C^{\dagger}(u) \int_{(t,t+s]} H_2(u) \,\d M_C^{\dagger}(u) 
        & = \int_{(t,t+s]} \int_{(t,t+s]} H_1(u) H_2(v) \,\d M_C^{\dagger}(u) \,\d M_C^{\dagger}(v) \nonumber \\
        & \hspace{-60mm} = \int_{(t,t+s]} \int_{(t,v)} H_1(u) H_2(v) \,\d M_C^{\dagger}(u) \,\d M_C^{\dagger}(v) 
        + \int_{(t,t+s]} \int_{\{v\}} H_1(u) H_2(v) \,\d M_C^{\dagger}(u) \,\d M_C^{\dagger}(v) \nonumber \\
        & \hspace{-45mm} + \int_{(t,t+s]} \int_{(v,t+s]} H_1(u) H_2(v) \,\d M_C^{\dagger}(u) \,\d M_C^{\dagger}(v). \label{eq:covar-3part}
    \end{align}

    We now simplify the conditional expectation of each of these three summands. Recall that we are assuming that censoring is discrete so that the identified cumulative hazard for censoring is
    \begin{equation*}
        \Lambda_C^{\dagger}(t)
        = \sum_{0 < u \leq t} \delta \Lambda_C^{\dagger}(u). 
    \end{equation*}

    The conditional expectation of the first summand in \cref{eq:covar-3part} may be seen to vanish by leveraging that $u<v$, that is, since
    \begin{align*}
        & \E \left\{ \int_{(t,t+s]} \int_{(t,v)} H_1(u) H_2(v) \,\d M_C^{\dagger}(u) \,\d M_C^{\dagger}(v) \,\middle|\, \F_t \right\} \\
        & \hspace{25mm} = \E \left\{ \sum_{t < v \leq t+s} \sum_{ t < u < v} H_1(u) H_2(v) \,\delta M_C^{\dagger}(u) \,\delta M_C^{\dagger}(v) \,\middle|\, \F_t \right\} \\
        & \hspace{25mm} = \E \left[ \sum_{t < v \leq t+s} \sum_{ t < u < v} \E \left\{ H_1(u) H_2(v) \,\delta M_C^{\dagger}(u) \,\delta M_C^{\dagger}(v) \,\middle|\, \G_v \right\} \,\middle|\, \F_t \right] \\ %
        & \hspace{25mm} = \E \left[ \sum_{t < v \leq t+s} \sum_{ t < u < v} H_1(u) H_2(v) \,\delta M_C^{\dagger}(u) \E \left\{ \delta M_C^{\dagger}(v)  \,\middle|\, \G_v \right\} \,\middle|\, \F_t \right] \\
        & \hspace{25mm} = 0. 
    \end{align*}
    Similarly, the conditional expectation of the last summand in \cref{eq:covar-3part} may be seen to vanish by leveraging that $u>v$, that is, since
    \begin{align*}
        & \E \left\{ \int_{(t,t+s]} \int_{(v, t+s]} H_1(u) H_2(v) \,\d M_C^{\dagger}(u) \,\d M_C^{\dagger}(v) \,\middle|\, \F_t \right\} \\
        & \hspace{25mm} = \E \left\{ \sum_{t < v \leq t+s} \sum_{v < u < t+s} H_1(u) H_2(v) \,\delta M_C^{\dagger}(u) \,\delta M_C^{\dagger}(v) \,\middle|\, \F_t \right\} \\
        & \hspace{25mm} = \E \left[ \sum_{t < v \leq t+s} \sum_{v < u < t+s} \E \left\{ H_1(u) H_2(v) \,\delta M_C^{\dagger}(u) \,\delta M_C^{\dagger}(v) \,\middle|\, \G_u \right\} \,\middle|\, \F_t \right] \\ 
        & \hspace{25mm} = \E \left[ \sum_{t < v \leq t+s} \sum_{v < u < t+s} H_1(u) H_2(v) \,\E \left\{ \delta M_C^{\dagger}(u) \,\middle|\, \G_u \right\} \,\delta M_C^{\dagger}(v) \,\middle|\, \F_t \right] \\
        & \hspace{25mm} = 0. 
    \end{align*}

    The conditional expectation of the middle summand in \cref{eq:covar-3part}
    \begin{align}
        \E \left\{ \int_{(t,t+s]} \int_{\{v\}} H_1(u) H_2(v) \,\d M_C^{\dagger}(u) \,\d M_C^{\dagger}(v) \,\middle|\, \F_t \right\} 
        & = \E \left\{ \int_{(t,t+s]} H_1(v) H_2(v) \,\delta M_C^{\dagger}(v) \,\d M_C^{\dagger}(v) \,\middle|\, \F_t \right\} \nonumber \\
        & \hspace{-30mm} = \E \left[ \sum_{t < v \leq t+s} H_1(v) H_2(v) \left\{ \,\delta M_C^{\dagger}(v) \right\}^2 \,\middle|\, \F_t \right], \label{eq:mid-sum}
    \end{align}
    where the last equality applies the discreteness assumption, 
    is more complicated to handle. Since
    \begin{align*}
        \delta \left\{ M_C^{\dagger}(v) M_C^{\dagger}(v) \right\}
        & = \left\{ N_C(v) - \sum_{0 < u \leq v} Y^{\dagger}(u) \delta \Lambda_C^{\dagger}(u) \right\}^2 - \left\{ N_C(v-) - \sum_{0 < u < v} Y^{\dagger}(u) \delta \Lambda_C^{\dagger}(u) \right\}^2 \\
        & \hspace{-15mm} = \left\{ \delta N_C(v) - Y^{\dagger}(v) \delta \Lambda_C^{\dagger}(v) \right\}^2 + 2 \left\{ \delta N_C(v) - Y^{\dagger}(v) \delta \Lambda_C^{\dagger}(v) \right\} \left\{ N_C(v-) - \sum_{0 < u < v} Y^{\dagger}(u) \delta \Lambda_C^{\dagger}(u) \right\} \\
        & \hspace{-15mm} = \left\{ \delta N_C(v) - Y^{\dagger}(v) \delta \Lambda_C^{\dagger}(v) \right\}^2 - 2 \left\{ \delta N_C(v) - Y^{\dagger}(v) \delta \Lambda_C^{\dagger}(v) \right\} \sum_{0 < u < v} Y^{\dagger}(u) \delta \Lambda_C^{\dagger}(u) \\
        & \hspace{-15mm} = \left\{ \delta M_C^{\dagger}(v) \right\}^2 - 2 \, \delta M_C^{\dagger}(v) \sum_{0 < u < v} Y^{\dagger}(u) \delta \Lambda_C^{\dagger}(u),
    \end{align*}
    we see that \cref{eq:mid-sum} equals 
    \begin{align}
        & \E \left[ \sum_{t < v \leq t+s} H_1(v) H_2(v) \left\{ \delta \left\{ M_C^{\dagger}(v) M_C^{\dagger}(v) \right\} + 2 \, \delta M_C^{\dagger}(v) \sum_{0 < u < v} Y^{\dagger}(u) \delta \Lambda_C^{\dagger}(u) \right\} \,\middle|\, \F_t \right] \nonumber \\
        & \hspace{15mm} = \E \left[ \sum_{t < v \leq t+s} H_1(v) H_2(v) \left\{ \delta \left\{ M_C^{\dagger}(v) M_C^{\dagger}(v) \right\} + 2 \, \E \left( \delta M_C^{\dagger}(v) \middle| \G_v \right) \sum_{0 < u < v} Y^{\dagger}(u) \delta \Lambda_C^{\dagger}(u) \right\} \,\middle|\, \F_t \right] \nonumber \\
        & \hspace{15mm} = \E \left[ \sum_{t < v \leq t+s} H_1(v) H_2(v) \,\delta \left\{ M_C^{\dagger}(v) M_C^{\dagger}(v) \right\} \,\middle|\, \F_t \right]. \label{eq:cent2}
    \end{align}
    
    Since \cref{eq:cent2} has $\left\{ M_C^{\dagger}(v) \right\}^2$ as an integrator and \Cref{prop:dag-square-mart} shows that 
    \begin{equation} 
        \left\{ M_C^{\dagger}(v) \right\}^2
        - \int_{(0,v]} \left\{ 1 - \delta \Lambda_C^{\dagger}(u) \right\} Y^{\dagger}(u) \,\d \Lambda_C^{\dagger}(u) \label{eq:cens-var-mart} 
    \end{equation} 
    is an $\F_v$-martingale, we now seek to express \cref{eq:cent2} in a way involving a martingale transform with respect to \cref{eq:cens-var-mart}. 
    Since the martingale transform satisfies
    \begin{align*}
        & \E \left[ \sum_{t < v \leq t+s} H_1(v) H_2(v) \biggl\langle \delta \left\{ M_C^{\dagger}(v) M_C^{\dagger}(v) \right\} - \left\{ 1 - \delta \Lambda_C^{\dagger}(v) \right\} Y^{\dagger}(v) \,\delta \Lambda_C^{\dagger}(v) \biggr\rangle \,\middle|\, \F_t \right] \\
        & \hspace{15mm} = \E \left[ \sum_{t < v \leq t+s} H_1(v) H_2(v) \E \biggl\langle \delta \left\{ M_C^{\dagger}(v) M_C^{\dagger}(v) \right\} - \left\{ 1 - \delta \Lambda_C^{\dagger}(v) \right\} Y^{\dagger}(v) \,\delta \Lambda_C^{\dagger}(v) \bigg| \G_v \biggr\rangle \,\middle|\, \F_t \right] \\
        & \hspace{15mm} = 0,
    \end{align*}
    by adding and subtracting we see that \cref{eq:cent2} equals
    \begin{align*}
        & \E \left[ \sum_{t < v \leq t+s} H_1(v) H_2(v) \left\{ 1 - \delta \Lambda_C^{\dagger}(v) \right\} Y^{\dagger}(v) \,\delta \Lambda_C^{\dagger}(v) \,\middle|\, \F_t \right] \\
        & \hspace{30mm} = \E \left[ \int_{(t, t+s]} H_1(v) H_2(v) \left\{ 1 - \delta \Lambda_C^{\dagger}(v) \right\} Y^{\dagger}(v) \,\delta \Lambda_C^{\dagger}(v) \,\middle|\, \F_t \right]
    \end{align*}

    In summary, we have shown that 
    \begin{align*}
        & \E \left\{ \int_{(0,t+s]} H_1(u) \,\d M_C^{\dagger}(u) \int_{(0,t+s]} H_2(u) \,\d M_C^{\dagger}(u) \,\middle|\, \F_t \right\} \nonumber \\
        =\, & \int_{(0,t]} H_1(u) \,\d M_C^{\dagger}(u) \int_{(0,t]} H_2(u) \,\d M_C^{\dagger}(u) 
        + \E \left\{ \int_{(t,t+s]} H_1(u) \,\d M_C^{\dagger}(u) \int_{(t,t+s]} H_2(u) \,\d M_C^{\dagger}(u) \,\middle|\, \F_t \right\} \\
        =\, & \int_{(0,t]} H_1(u) \,\d M_C^{\dagger}(u) \int_{(0,t]} H_2(u) \,\d M_C^{\dagger}(u) 
        + \E \left[ \int_{(t, t+s]} H_1(v) H_2(v) \left\{ 1 - \delta \Lambda_C^{\dagger}(v) \right\} Y^{\dagger}(v) \,\delta \Lambda_C^{\dagger}(v) \,\middle|\, \F_t \right],
    \end{align*}
    proving that the stochastic process is drift free. It is therefore a martingale.

    We now turn our attention to the covariation process. %
    Adaptedness and integrability are straightforward to verify, so we focus on proving the no-drift property. The key steps of the proof are the same as in the previous case for the variation process; however, we go through all of the details again for completeness. 
    
    Let $s,t>0$. By expanding the integrals over disjoint intervals $(0,t]$ and $(t, t+s]$, we can show that 
    \begin{align}
        & \E \left\{ \int_{(0,t+s]} H_1(u) \,\d M_C^{\dagger}(u) \int_{(0,t+s]} H_3(u) \,\d M_T^{\sharp}(u) \,\middle|\, \F_t \right\} \nonumber \\
        =\, & \int_{(0,t]} H_1(u) \,\d M_C^{\dagger}(u) \int_{(0,t]} H_3(u) \,\d M_T^{\sharp}(u) 
        + \E \left\{ \int_{(t,t+s]} H_1(u) \,\d M_C^{\dagger}(u) \int_{(t,t+s]} H_3(u) \,\d M_T^{\sharp}(u) \,\middle|\, \F_t \right\}, \label{eq:f-covar-cent-basic}
    \end{align}
    since each integral is a martingale so is mean zero conditional on $\F_t$. 
    We may determine the desired centering term by expressing the conditional expectation in \cref{eq:f-covar-cent-basic} as a contrast between the centering term evaluated at $t+s$ and at $t$. Note, if the centering term is an integral over $(0, \cdot]$, then this contrast will simply be an integral over $(t, t+s]$. 

    By expanding the interval $(t,t+s] = (t,v) \cup \{v\} \cup (v, t+s]$, the double integral in the conditional expectation may be written as
    \begin{align}
        \int_{(t,t+s]} H_1(u) \,\d M_C^{\dagger}(u) \int_{(t,t+s]} H_3(u) \,\d M_T^{\sharp}(u) 
        & = \int_{(t,t+s]} \int_{(t,t+s]} H_1(u) H_3(v) \,\d M_C^{\dagger}(u) \,\d M_T^{\sharp}(v) \nonumber \\
        & \hspace{-60mm} = \int_{(t,t+s]} \int_{(t,v)} H_1(u) H_3(v) \,\d M_C^{\dagger}(u) \,\d M_T^{\sharp}(v) 
        + \int_{(t,t+s]} \int_{\{v\}} H_1(u) H_3(v) \,\d M_C^{\dagger}(u) \,\d M_T^{\sharp}(v) \nonumber \\
        & \hspace{-45mm} + \int_{(t,t+s]} \int_{(v,t+s]} H_1(u) H_3(v) \,\d M_C^{\dagger}(u) \,\d M_T^{\sharp}(v). \label{eq:f-covar-3part}
    \end{align}
    
    We now simplify the conditional expectation of each of these three summands. Recall that we are assuming that the time at-risk is discrete so that the identified cumulative hazards are
    \begin{equation*}
        \Lambda_C^{\dagger}(t)
        = \sum_{0 < u \leq t} \delta \Lambda_C^{\dagger}(u); \hspace{5mm}
        \Lambda_T^{\sharp}(t)
        = \sum_{0 < u \leq t} \delta \Lambda_T^{\sharp}(u). 
    \end{equation*}
    
    The conditional expectation of the first summand in \cref{eq:covar-3part} may be seen to vanish by leveraging that $u<v$, that is, since
    \begin{align*}
        & \E \left\{ \int_{(t,t+s]} \int_{(t,v)} H_1(u) H_3(v) \,\d M_C^{\dagger}(u) \,\d M_T^{\sharp}(v) \,\middle|\, \F_t \right\} \\
        & \hspace{25mm} = \E \left\{ \sum_{t < v \leq t+s} \sum_{ t < u < v} H_1(u) H_2(v) \,\delta M_C^{\dagger}(u) \,\delta M_T^{\sharp}(v) \,\middle|\, \F_t \right\} \\
        & \hspace{25mm} = \E \left[ \sum_{t < v \leq t+s} \sum_{ t < u < v} \E \left\{ H_1(u) H_2(v) \,\delta M_C^{\dagger}(u) \,\delta M_T^{\sharp}(v) \,\middle|\, \F_{v-} \right\} \,\middle|\, \F_t \right] \\
        & \hspace{25mm} = \E \left[ \sum_{t < v \leq t+s} \sum_{ t < u < v} H_1(u) H_2(v) \,\delta M_C^{\dagger}(u) \E \left\{ \delta M_T^{\sharp}(v)  \,\middle|\, \F_{v-} \right\} \,\middle|\, \F_t \right] \\
        & \hspace{25mm} = 0, 
    \end{align*}
    where the second equality uses that $H_1(u) \in \G_u \subseteq \F_{v-}$, $\delta M_C^{\dagger}(u) \in \F_u \subseteq \F_{v-}$ since $u<v$, and that $H_3(v) \in \F_{v-}$ by \citet[Prop. 1.4.1]{fleming1991counting}. 

    Similarly, the conditional expectation of the last summand in \cref{eq:covar-3part} may be seen to vanish by leveraging that $u>v$, that is, since 
    \begin{align*}
        & \E \left\{ \int_{(t,t+s]} \int_{(v, t+s]} H_1(u) H_3(v) \,\d M_C^{\dagger}(u) \,\d M_T^{\sharp}(v) \,\middle|\, \F_t \right\} \\
        & \hspace{25mm} = \E \left\{ \sum_{t < v \leq t+s} \sum_{v < u < t+s} H_1(u) H_3(v) \,\delta M_C^{\dagger}(u) \,\delta M_T^{\sharp}(v) \,\middle|\, \F_t \right\} \\
        & \hspace{25mm} = \E \left[ \sum_{t < v \leq t+s} \sum_{v < u < t+s} \E \left\{ H_1(u) H_3(v) \,\delta M_C^{\dagger}(u) \,\delta M_T^{\sharp}(v) \,\middle|\, \G_u \right\} \,\middle|\, \F_t \right] \\ 
        & \hspace{25mm} = \E \left[ \sum_{t < v \leq t+s} \sum_{v < u < t+s} H_1(u) H_3(v) \,\E \left\{ \delta M_C^{\dagger}(u) \,\middle|\, \G_u \right\} \,\delta M_T^{\sharp}(v) \,\middle|\, \F_t \right] \\ 
        & \hspace{25mm} = 0. 
    \end{align*}

    The conditional expectation of the middle summand in \cref{eq:covar-3part}
    \begin{align}
        \E \left\{ \int_{(t,t+s]} \int_{\{v\}} H_1(u) H_3(v) \,\d M_C^{\dagger}(u) \,\d M_T^{\sharp}(v) \,\middle|\, \F_t \right\} 
        & = \E \left\{ \int_{(t,t+s]} H_1(v) H_3(v) \,\delta M_C^{\dagger}(v) \,\d M_T^{\sharp}(v) \,\middle|\, \F_t \right\} \nonumber \\
        & \hspace{-30mm} = \E \left[ \sum_{t < v \leq t+s} H_1(v) H_3(v) \,\delta M_C^{\dagger}(v) \,\delta M_T^{\sharp}(v) \,\middle|\, \F_t \right], \label{eq:f-mid-sum}
    \end{align}
    where the last equality applies the discreteness assumption, 
    is more complicated to handle. Since
    \begin{align*}
        \delta M_C^{\dagger}(v) \delta M_T^{\sharp}(v)
        & = \left\{ \delta N_C(v) - Y^{\dagger}(v) \delta \Lambda_C^{\dagger}(v) \right\} \left\{ \delta N_T(v) - Y^{\sharp}(v) \delta \Lambda_T^{\sharp}(v) \right\} \\
        & \hspace{-15mm} = - \delta \Lambda_T^{\sharp}(v) \left\{ \delta N_C(v) - Y^{\dagger}(v) \delta \Lambda_C^{\dagger}(v) \right\}
    \end{align*}
    we see that \cref{eq:f-mid-sum} equals 
    \begin{align}
        - \E \left[ \sum_{t < v \leq t+s} H_1(v) H_3(v) \delta \Lambda_T^{\sharp}(v) \delta M_C^{\dagger}(v) \,\middle|\, \F_t \right]
        = 0, \label{eq:f-cent2}
    \end{align}
    by \Cref{thm:martingale-transform} since $\delta \Lambda_T^{\sharp}(v) \leq 1$ is bounded. 

    In summary, we have shown that 
    \begin{equation*}
        \E \left\{ \int_{(0,t+s]} H_1(u) \,\d M_C^{\dagger}(u) \int_{(0,t+s]} H_3(u) \,\d M_T^{\sharp}(u) \,\middle|\, \F_t \right\}
        = \int_{(0,t]} H_1(u) \,\d M_C^{\dagger}(u) \int_{(0,t]} H_3(u) \,\d M_T^{\sharp}(u),
    \end{equation*}
    proving that the product of these martingale transforms is itself a martingale. 
\end{proof}

\section{Additional Material for \texorpdfstring{\cref{sec:pred}}{Section 4}}
\label{sec:supp:pred-addntl}

\subsection{Proofs}
\label{sec:supp:pred-addntl:proofs}

Throughout the proofs in this section we adopt the following helpful notation to study tail behavior. Define ${\cal S}^- = \{s: \E[Y^{\sharp}(s+)] >0 \}$, 
${\cal S} = \{s: \E[Y^{\dagger}(s)] >0 \}$,
and ${\cal S}^+ = \{s: \E[Y^{\sharp}(s)] >0 \}$. Then, because 
$Y^{\sharp}(s+) \leq Y^{\dagger}(s) \leq Y^{\sharp}(s)$ for $s \geq 0$,
it can be shown that ${\cal S}^-
\subseteq {\cal S} \subseteq 
{\cal S}^+$. 
Define $\tau = 
\sup {\cal S}^+;$ then, 
because $\sup {\cal S}^+ = \sup {\cal S}^-$,
it follows that 
$\tau = \sup {\cal S}$.

\begin{proof}[Proof of \cref{lem:dag-dm}]
    Consider first the case where $t \in S$. Then, for $t < \tau$ and
    using the definitions of $M_C^{\dagger}(t)$ and $M_C^{\sharp}(t)$
    and the fact that $Y^{\dagger}(u) = Y^{\sharp}(u) - \delta N_T(u)$,
    \begin{align*}
        M_C^{\dagger}(t) - M_C^{\sharp}(t) 
        & = \int_{(0,t]} Y^{\sharp}(u) \,\d \Lambda_C^{\sharp}(u) - \int_{(0,t]} Y^{\sharp}(u) \,\d \Lambda_C^{\dagger}(u) + \int_{(0,t]} \delta N_T(u) \,\d \Lambda_C^{\dagger}(u) \\
        & \hspace{-10mm} = - \int_{(0,t]} Y^{\sharp}(u) \delta \Lambda_T^{\sharp}(u) \,\d \Lambda_C^{\dagger}(u) + \int_{(0,t]} \delta N_T(u) \,\d \Lambda_C^{\dagger}(u) \\
        & \hspace{-10mm} = \int_{(0,t]} \delta M_T^{\sharp}(u) \,\d \Lambda_C^{\dagger}(u) \\
        & \hspace{-10mm} = \int_{(0,t]} \delta \Lambda_C^{\dagger}(u) \,\d M_T^{\sharp}(u),
    \end{align*}
    where the second equality follows from the useful identity
    \begin{equation}
        \frac{\E \{ Y^{\dagger}(u) \}}{\E \{ Y^{\sharp}(u) \}} - 1
        = \frac{\E \{Y^{\dag}(u) \} - \E \{Y^{\sharp}(u) \}}{\E \{Y^{\sharp}(u) \}}
        = \frac{- \E \{ \delta N_T(u) \}}{\E \{Y^{\sharp}(u) \}}
        = - \delta \Lambda_T^{\sharp}(u). \label{eq:Ru} 
    \end{equation}

    Next, suppose $t \geq \tau$. Using the result above, we can decompose $M_C^{\dagger}(t)$ as follows:
    \begin{equation}
    \label{eq:MCdag-decomp}
    M_C^{\dagger}(t) = 
    \Bigl\{
    M_C^{\sharp}(\tau-) +
    \int_{(0,\tau)} \delta \Lambda_C^{\dagger}(u) \,\d M_T^{\sharp}(u)
    \Bigr\}
    +
    \delta M_C^{\dagger}(\tau)
    +    
    \int_{(\tau,t]} \d
    M_C^{\dagger}(t).
    \end{equation}
    Now, for $u > \tau$, we necessarily have
    $\delta N_T(u) = 0$ and $Y^\sharp(u) = 0;$ 
    moreover, by arguments given
    in \cref{sec:supp:basic-addntl:props}
    (see also \cref{cor:lambda-sharp-coerc}), 
    we also have $\delta \Lambda_C^{\dagger}(u) = 0$
    on this same interval. 
    Hence,    
    \[
    \int_{(\tau,t]} \d
    M_C^{\dagger}(t) = 0,
    \]
    and it suffices to consider
    the behavior of $\delta M_C^{\dagger}(\tau)$. 
    
    Similarly to \cref{cor:lambda-sharp-coerc}, it is again useful to separately consider the cases that $\tau \in {\cal S}^c$ and $\tau \in {\cal S}$. Suppose first that $\tau \in {\cal S}^c$. Then, the 
    analysis just done for $u > \tau$ continues to apply;
    hence, $\delta M_C^{\dagger}(\tau)  
    = \delta M_C^{\sharp}(\tau) = 0$
    and it follows  that
    \[
    M_C^{\dagger}(t) = 
    M_C^{\sharp}(\tau) +
    \int_{(0,\tau]} \delta \Lambda_C^{\dagger}(u) \,\d M_T^{\sharp}
    (u)
    = 
    M_C^{\sharp}(t) +
    \int_{(0,t]} \delta \Lambda_C^{\dagger}(u) \,\d M_T^{\sharp}(u),
    \]
    establishing the desired result
    in this case.
    Now, suppose instead that $\tau \in {\cal S};$
    in this case, algebra shows
    \begin{align*}
    \delta M_C^{\dagger}(\tau) & = 
    \delta N_C^{\dagger}(\tau) -
    Y^\dagger(\tau) \delta \Lambda_C^{\dagger}(\tau) \\
    & = 
    \delta M_C^{\sharp}(\tau)
    + Y^\sharp(\tau) \bigl\{
    \delta \Lambda_C^{\sharp}(\tau)
    -  \delta \Lambda_C^{\dagger}(\tau) \bigr\}
    + \delta N_T(\tau) 
     \delta \Lambda_C^{\dagger}(\tau).
    \end{align*}
    Given that $\tau \in {\cal S}$,
    we have $
    \E \{ Y^{\sharp}(\tau) \}  \geq 
    \E \{ Y^{\dagger}(\tau) \} > 0;$
    therefore,
    \begin{align*}
    \delta M_C^{\dagger}(\tau) 
     & = 
    \delta M_C^{\sharp}(\tau)
    + Y^\sharp(\tau) \Bigl\{
    \frac{\E \{ Y^{\dagger}(u) \}}{\E \{ Y^{\sharp}(u) \}} - 1 \Bigr\}
    \delta \Lambda_C^{\dagger}(\tau) 
    + \delta N_T(\tau) 
     \delta \Lambda_C^{\dagger}(\tau) \\
     & = 
     \delta M_C^{\sharp}(\tau)
     - Y^\sharp(\tau) \delta \Lambda_T^{\sharp}(\tau) \delta \Lambda_C^{\dagger}(\tau) 
    + \delta N_T(\tau) 
     \delta \Lambda_C^{\dagger}(\tau) \\
    & = \delta M_C^{\sharp}(\tau)
     + \delta M^{\sharp}_T(\tau) 
     \delta \Lambda_C^{\dagger}(\tau), 
    \end{align*}
    the penultimate step following from 
    \eqref{eq:Ru}. Returning to
    \eqref{eq:MCdag-decomp},
    \begin{align*}
   M_C^{\dagger}(t) & = 
    \biggl\{
    M_C^{\sharp}(\tau-) +
    \int_{(0,\tau)} \delta \Lambda_C^{\dagger}(u) \,\d M_T^{\sharp}(u)
    \biggr\}
    +\delta M_C^{\sharp}(\tau) +
    \delta M^{\sharp}_T(\tau) 
     \delta \Lambda_C^{\dagger}(\tau) \\
     & =
     M_C^{\sharp}(\tau) +
    \int_{(0,\tau]} \delta \Lambda_C^{\dagger}(u) \,\d M_T^{\sharp}(u) \\
    & =
     M_C^{\sharp}(t) +
    \int_{(0,t]} \delta \Lambda_C^{\dagger}(u) \,\d M_T^{\sharp}(u)
    \end{align*}
    as desired.
\end{proof}

\begin{proof}[Proof of \cref{prop:pred-trans}]
    Applying \cref{lem:dag-dm}, the martingale transform is 
    \begin{equation}
        (H \boldsymbol{\cdot} M^{\dagger}_{C})
        = \int_{(0,t]} H(u) \,\d M^{\dagger}_{C}(u) 
        = \int_{(0,t]} H(u) \,\d M^{\sharp}_{C}(u) + \int_{(0,t]} H(u) \delta \Lambda_C^{\dagger} \,\d M^{\sharp}_{T}(u)
        = (H \boldsymbol{\cdot} M^{\sharp}_{C}) + (\delta \Lambda_C^{\dagger} \, H \boldsymbol{\cdot} M^{\sharp}_{T}).
    \end{equation}
    Each summand is a $\F_t$-martingale by Theorem 2.4.4 in \citet{fleming1991counting}, since $\delta \Lambda_C^{\dagger}$ is bounded and non-stochastic and the first assumptions of \cref{prop:pred-trans} correspond to the assumptions of Theorem 2.4.4. The result follows since a sum of martingales is a martingale. 
\end{proof}

Denote $\langle M \rangle$ as the $\F_t$-predictable variation process of an $\F_t$-martingale $M$. 

\begin{proof}[Proof of \cref{prop:pred-covar}]
    We calculate the $\F_t$-predictable variation $\langle M^{\dagger}_{C} \rangle$ and the $\F_t$-predictable covariation $\langle M^{\dagger}_{C}, M^{\sharp}_{T} \rangle$. After, the result follows directly from Theorem 2.4.3 in \citet{fleming1991counting}.

    We start by calculating the $\F_t$-predictable covariation $\langle M^{\dagger}_{C}, M^{\sharp}_{T} \rangle$. 
    Consider the stochastic process
    \begin{align*}
        M_C^{\dagger}(t) M_T^{\sharp}(t)
        & = \left\{ M^\sharp_C(t) + \int_{(0,t]} \delta \Lambda^\dag_C(u) \,\d M_T^{\sharp}(u) \right\} M_T^{\sharp}(t) \\
        & = M^\sharp_C(t) M^\sharp_T(t) + M_T^{\sharp}(t) \int_{(0,t]} \delta \Lambda^\dag_C(u) \,\d M_T^{\sharp}(u). 
    \end{align*}
    We calculate the compensator for each summand using Theorem 2.4.2 and 2.6.1 in \citet{fleming1991counting}. First we see that 
    \begin{equation*}
        \Bigl\langle M^\sharp_C(t), M^\sharp_T(t) \Bigr\rangle
        = -  \int_{(0,t]} Y^{\sharp}(u) \delta \Lambda_T^{\sharp}(u) \,\d \left\{ \int_{(0,u]} Y^{\sharp}(v) \,\d \Lambda_C^{\sharp}(v) \right\} 
        = - \int_{(0,t]} Y^{\sharp}(u) \delta \Lambda_T^{\sharp}(u) \,\d \Lambda_C^{\sharp}(u). 
    \end{equation*}
    Second, 
    \begin{equation*}
        \Bigl\langle M_T^{\sharp}(t), \int_{(0,t]} \delta \Lambda^\dag_C(u) \,\d M_T^{\sharp}(u) \Bigr\rangle
        = \int_{(0,t]} \delta \Lambda^\dag_C(u) \{ 1 - Y^{\sharp}(u) \delta\Lambda_T^{\sharp}(u) \} Y^{\sharp}(u) \,\d \Lambda_T^{\sharp}(u).
    \end{equation*}
    Therefore the following is a $\F_t$-martingale,
    \begin{align*}
        & M_C^{\dagger}(t) M_T^{\sharp}(t)
        + \int_{(0,t]} Y^{\sharp}(u) \delta \Lambda_T^{\sharp}(u) \,\d \Lambda_C^{\sharp}(u)
        - \int_{(0,t]} Y^{\sharp}(u) \delta \Lambda^\dag_C(u) \{ 1 - \delta\Lambda_T^{\sharp}(u) \} \,\d \Lambda_T^{\sharp}(u) \\
        & \hspace{10mm} = M_C^{\dagger}(t) M_T^{\sharp}(t) + \int_{(0,t]} \left\{ \delta \Lambda_C^{\sharp}(u) - \delta \Lambda^\dag_C(u) \{ 1 - \delta\Lambda_T^{\sharp}(u) \} \right\} Y^{\sharp}(u) \,\d \Lambda_T^{\sharp}(u) \\
        & \hspace{10mm} = M_C^{\dagger}(t) M_T^{\sharp}(t),
    \end{align*}
    where the last equality follows by applying \cref{eq:Ru}.

    We now calculate the $\F_t$-predictable variation $\langle M^{\dagger}_{C} \rangle$. 
    Consider the stochastic process
    \begin{align*}
        M_C^{\dagger}(t) M_C^{\dagger}(t)
        & = \left\{ M^\sharp_C(t) + \int_{(0,t]} \delta \Lambda^\dag_C(u) \,\d M_T^{\sharp}(u) \right\}^2 \\
        & \hspace{-10mm} = M^\sharp_C(t) M^\sharp_C(t) + 2 M_C^{\sharp}(t) \int_{(0,t]} \delta \Lambda^\dag_C(u) \,\d M_T^{\sharp}(u) + \left\{ \int_{(0,t]} \delta \Lambda^\dag_C(u) \,\d M_T^{\sharp}(u) \right\}^2. 
    \end{align*}
    We calculate the compensator for each summand using Theorem 2.4.2 and 2.6.1 in \citet{fleming1991counting}. First we see that 
    \begin{equation*}
        \Bigl\langle M^\sharp_C(t) \Bigr\rangle
        = \int_{(0,t]} \{ 1 - \delta \Lambda_C^{\sharp}(u) \} Y^{\sharp}(u) \,\d \Lambda_C^{\sharp}(u)
        = \int_{(0,t]} \{ 1 - \delta \Lambda_C^{\sharp}(u) \} \{ 1 - \delta \Lambda_T^{\sharp}(u) \} Y^{\sharp}(u) \,\d \Lambda_C^{\dagger}(u),
    \end{equation*}
    by applying \cref{eq:Ru}; 
    second, 
    \begin{equation*}
        \Bigl\langle M_C^{\sharp}(t), \int_{(0,t]} \delta \Lambda^\dag_C(u) \,\d M_T^{\sharp}(u) \Bigr\rangle
        = \int_{(0,t]} \{ \delta \Lambda^\dag_C(u) \}^2 \{ 1 - \delta \Lambda_T^{\sharp}(u) \} Y^{\sharp}(u) \,\d \Lambda_T^{\sharp}(u);
    \end{equation*}
    third,
    \begin{equation*}
        \Bigl\langle \int_{(0,t]} \delta \Lambda^\dag_C(u) \,\d M_T^{\sharp}(u) \Bigr\rangle
        = \int_{(0,t]} \delta \Lambda^\dag_C(u) \delta \Lambda^\sharp_C(u) Y^{\sharp}(u) \,\d \Lambda_T^{\sharp}(u). 
    \end{equation*}
    Therefore, due to the bilinearity of $\F_t$-predictable variation, the $\F_t$-predictable variation of $M_C^{\dagger}$ is 
    \begin{align*}
        \langle M_C^{\dagger}(t) \rangle
        & = \left\langle M^\sharp_C(t) + \int_{(0,t]} \delta \Lambda^\dag_C(u) \,\d M_T^{\sharp}(u) \right\rangle \\
        & = \left\langle M^\sharp_C(t) \right\rangle 
        + \left\langle \int_{(0,t]} \delta \Lambda^\dag_C(u) \,\d M_T^{\sharp}(u) \right\rangle 
        + 2 \left\langle M^\sharp_C(t), \int_{(0,t]} \delta \Lambda^\dag_C(u) \,\d M_T^{\sharp}(u) \right\rangle \\
        & = \int_{(0,t]} \{ 1 - \delta \Lambda_C^{\sharp}(u) \} \{ 1 - \delta \Lambda_T^{\sharp}(u) \} Y^{\sharp}(u) \,\d \Lambda_C^{\dagger}(u) 
        + \int_{(0,t]} \delta \Lambda^\dag_C(u) \delta \Lambda_T^{\sharp}(u) \{ 1 - \delta \Lambda_T^{\sharp}(u) \} Y^{\sharp}(u) \,\d \Lambda_C^{\dagger}(u) \\
        & \hspace{25mm} - 2 \int_{(0,t]} \delta \Lambda^\sharp_T(u) \delta \Lambda^\sharp_C(u) Y^{\sharp}(u) \,\d \Lambda_C^{\dag}(u) \\
        & = \int_{(0,t]} \left[ \{ 1 - \delta \Lambda_C^{\sharp}(u) \} \{ 1 - \delta \Lambda_T^{\sharp}(u) \} + \delta \Lambda^\dag_C(u) \delta \Lambda_T^{\sharp}(u) \{ 1 - \delta \Lambda_T^{\sharp}(u) \} - 2 \delta \Lambda^\sharp_T(u) \delta \Lambda^\sharp_C(u) \right] Y^{\sharp}(u) \,\d \Lambda_C^{\dag}(u) \\
        & = \int_{(0,t]} \frac{F^{\sharp}(u) G^{\dagger}(u)}{F^{\sharp}(u-) G^{\dagger}(u-)} Y^{\sharp}(u) \,\d \Lambda_C^{\dag}(u), 
    \end{align*}
    where last equality follows from the simplification
    \begin{align*}
        & \{ 1 - \delta \Lambda_C^{\sharp}(u) \} \{ 1 - \delta \Lambda_T^{\sharp}(u) \} + \delta \Lambda^\dag_C(u) \delta \Lambda_T^{\sharp}(u) \{ 1 - \delta \Lambda_T^{\sharp}(u) \} - 2 \delta \Lambda^\sharp_T(u) \delta \Lambda^\sharp_C(u) \\
        & \hspace{15mm} = \{ 1 - \delta \Lambda_C^{\sharp}(u) \} \{ 1 - \delta \Lambda_T^{\sharp}(u) \} + \delta \Lambda^\sharp_C(u) \delta \Lambda_T^{\sharp}(u)  - 2 \delta \Lambda^\sharp_T(u) \delta \Lambda^\sharp_C(u) \\
        & \hspace{15mm} = 1 - \delta \Lambda_C^{\sharp}(u) - \delta \Lambda_T^{\sharp}(u) \\
        & \hspace{15mm} = 1 - \delta \Lambda_C^{\dagger}(u) \{ 1 - \delta \Lambda_T^{\sharp}(u) \} - \delta \Lambda_T^{\sharp}(u) \\
        & \hspace{15mm} =  \{1  - \delta \Lambda^\sharp_T(u) \} \{1  - \delta \Lambda^\dag_C(u) \} \\
        & \hspace{15mm} = \frac{F^{\sharp}(u) G^{\dagger}(u)}{F^{\sharp}(u-) G^{\dagger}(u-)},
    \end{align*}
    which twice applied \cref{eq:Ru}. 
    Finally, the martingale transform results follow directly from Theorem 2.4.3 in \citet{fleming1991counting}. 
\end{proof}

The next result considers the difference between the $\F_t$-predictable centering terms, derived in \cref{prop:pred-covar}, and the $\G_t$-adapted centering terms, stated in \cref{thm:covar-proc}. The result is strictly algebraic, analogous to \cref{lem:dag-dm}. 

\begin{lem}
\label{prop:pred-covar2}
    For $t>0$, the difference between \cref{eq:half-pred-var} and variation martingale in \cref{eq:pred-var} is
    \begin{equation*}
        \int_{(0,t]} H_1(u) H_2(u) \delta \Lambda_C^{\dagger}(u) Y^{\sharp}(u) \frac{G^{\dagger}(u)}{G^{\dagger}(u-)} \,\d M_T^{\sharp}(u). 
    \end{equation*}
\end{lem}

\begin{proof}
    For simplicity we consider the difference in $\F_t$-predictable variation, i.e. so that $H_1, H_2 = 1$. For $t < \tau$, the difference is
    \begin{align*}
        & \left\{  - \int_{(0,t]} \frac{G^{\dagger}(u)}{G^{\dagger}(u-)} Y^{\dagger}(u) \,\d \Lambda_C^{\dagger}(u) \right\}
        - \left\{ - \int_{(0,t]} \frac{F^{\sharp}(u) G^{\dagger}(u)}{F^{\sharp}(u-) G^{\dagger}(u-)} Y^{\sharp}(u) \,\d \Lambda_C^{\dagger}(u) \right\} \\
        & \hspace{10mm} = \int_{(0,t]} \left\{ \frac{F^{\sharp}(u)}{F^{\sharp}(u-)} Y^{\sharp}(u) -  Y^{\dagger}(u) \right\} \frac{G^{\dagger}(u)}{G^{\dagger}(u-)} \,\d \Lambda_C^{\dagger}(u) \\
        & \hspace{10mm} = \int_{(0,t]} \left\{ \frac{F^{\sharp}(u)}{F^{\sharp}(u-)} - \{ 1 - \delta N_T(u) \} \right\} Y^{\sharp}(u) \frac{G^{\dagger}(u)}{G^{\dagger}(u-)} \,\d \Lambda_C^{\dagger}(u) \\
        & \hspace{10mm} = \int_{(0,t]} \left\{ \{ 1 - \delta \Lambda_T^{\sharp}(u) \} - \{ 1 - \delta N_T(u) \} \right\} Y^{\sharp}(u) \frac{G^{\dagger}(u)}{G^{\dagger}(u-)} \,\d \Lambda_C^{\dagger}(u) \\
        & \hspace{10mm} = \int_{(0,t]} \left\{ \delta N_T(u) - \delta \Lambda_T^{\sharp}(u) \right\} Y^{\sharp}(u) \frac{G^{\dagger}(u)}{G^{\dagger}(u-)} \,\d \Lambda_C^{\dagger}(u) \\
        & \hspace{10mm} = \int_{(0,t]} \delta \Lambda_C^{\dagger}(u) Y^{\sharp}(u) \frac{G^{\dagger}(u)}{G^{\dagger}(u-)} \,\d M_T^{\sharp}(u). 
    \end{align*}
    The result follows as the $t \geq \tau$ case follows similarly. 
\end{proof}

\begin{proof}[Proof of \cref{thm:pred-covar-proc}]
    That $(H_1 \boldsymbol{\cdot} M^{\dagger}_{C})(t) \, (H_2 \boldsymbol{\cdot} M^{\sharp}_{T})(t)$ is a local $\F_t$-martingale follows immediately from \cref{prop:pred-covar}; notice that $0$ is both a $\G_t$-adapted and a $\F_t$-predictable centering term. 

    Now, it follows from \cref{prop:pred-covar2} that 
    \begin{align*}
        & (H_1 \boldsymbol{\cdot} M^{\dagger}_{C})(t) \, (H_2 \boldsymbol{\cdot} M^{\dagger}_{C})(t) - \int_{(0,t]} H_1(u) H_2(u) \frac{F^{\sharp}(u)}{F^{\sharp}(u-)} Y^{\dagger}(u) \,\d \Lambda_C^{\dagger}(u) \\
        & \hspace{-5mm} = \left\{ (H_1 \boldsymbol{\cdot} M^{\dagger}_{C})(t) \, (H_2 \boldsymbol{\cdot} M^{\dagger}_{C})(t) - \int_{(0,t]} H_1(u) H_2(u) \frac{F^{\sharp}(u) G^{\dagger}(u)}{F^{\sharp}(u-) G^{\dagger}(u-)} Y^{\sharp}(u) \,\d \Lambda_C^{\dagger}(u) \right\} \\
        & \hspace{20mm} + \int_{(0,t]} H_1(u) H_2(u) \delta \Lambda_C^{\dagger}(u) Y^{\sharp}(u) \frac{G^{\dagger}(u)}{G^{\dagger}(u-)} \,\d M_T^{\sharp}(u).
    \end{align*}
    The first expression is a local $\F_t$-martingale by \cref{prop:pred-covar}, and the second expression is a local $\F_t$-martingale by Theorem 2.4.1 %
    in \citet{fleming1991counting}. The result follows since a sum of local martingales is a local martingale. 
\end{proof}

\subsection{Alternative proofs of previous results}
\label{sec:supp:pred-addntl:alt-proofs}

\begin{proof}[Alternative proof of \cref{prop:dagger-martingale}]
    By \cref{prop:sharp-martingale}, $M_T^{\sharp}$ and $M_C^{\sharp}$ are $\F_t$-martingales. Since $\delta \Lambda^\dag_C$ is non-stochastic and bounded by 1, we see $\int_{(0,\cdot]} \delta \Lambda_C^{\dagger}(u) \,\d M_T^{\sharp}(u)$ is a $\F_t$-martingale \citep[Theorem 1.5.1]{fleming1991counting}. Thus $M_C^{\sharp} +
    \int_{(0,\cdot]} \delta \Lambda_C^{\dagger}(u) \,\d M_T^{\sharp}(u)$ is an $\F_t$-martingale; the result follows from \cref{lem:dag-dm}. 
\end{proof}

\begin{proof}[Alternative proof of \cref{cor:marting-sharp-coerc}]
    Recall that the stochastic processes $M_C^{\dagger}$ and $M_C^{\sharp}$ are indistinguishable if their difference vanishes uniformly in $t$,  that is, if 
    \begin{equation*}
        \P \left\{ M_C^{\dagger}(t) - M_C^{\sharp}(t) = 0 \text{ for all } t>0 \right\}
        = 1.
    \end{equation*}

    The result in \cref{lem:dag-dm} 
    implies $M_C^{\dagger}(t) - M_C^{\sharp}(t) = \int_{(0,t]} \delta \Lambda_C^{\dagger}(u) \,\d M_T^{\sharp}(u)$. In view of the proofs of 
    \cref{cor:lambda-sharp-coerc} and
    \cref{lem:dag-dm}, it suffices
    to    prove that the martingale
       \begin{equation}
         \label{eq:mart-diff}  
   \int_{(0,t]} \delta \Lambda_C^{\dagger}(u) \,\d M_T^{\sharp}(u) = 0
       \end{equation}
 if and only if    
    \begin{equation}
    \label{eq:no-jumps}
    \sum_{0<u \leq t} \delta \Lambda_C^{\sharp}(u)
    \delta \Lambda_T^{\sharp}(u) = 0
    \end{equation}
    for $t \in (0,\tau]$.
    Equivalently, since \eqref{eq:mart-diff} holds if and only if
      \begin{equation}
         \label{eq:mart-diff-quadvar}  
    \left\langle 
    \int_{(0,\tau]} \delta \Lambda_C^{\dagger}(u) \,\d M_T^{\sharp}(u)
    \right\rangle = 0,
    \end{equation}
    by \citet[p. 186]{he1992semimartingale},
    we need only show that \eqref{eq:no-jumps}
    and \eqref{eq:mart-diff-quadvar}  
    are equivalent.
        We have
    \begin{align*}
    \left\langle 
    \int_{(0,\tau]} \delta \Lambda_C^{\dagger}(u) \,\d M_T^{\sharp}(u)
    \right\rangle 
    & = 
    \int_{(0,\tau]} \Bigl( \delta \Lambda_C^{\dagger}(u)
    \Bigr)^2 Y^{\sharp}(u) \,\d \Lambda_T^{\sharp}(u) \\
    & = 
     \sum_{0 < u \leq \tau} \Bigl( \delta \Lambda_C^{\dagger}(u)
    \Bigr)^2 Y^{\sharp}(u) \delta \Lambda_T^{\sharp}(u) \\
   & = (A) + 2 (B) + (C)
      \end{align*}
      where
    \begin{align*}
  (A) & = \sum_{0 < u \leq \tau}Y^{\sharp}(u)  \Bigl( \delta \Lambda_C^{\dagger}(u)
     -\delta \Lambda_C^{\sharp}(u)\Bigr)^2 \delta \Lambda_T^{\sharp}(u) \\
      (B) & = \sum_{0 < u \leq \tau}Y^{\sharp}(u)  \Bigl( \delta \Lambda_C^{\dagger}(u)
     -\delta \Lambda_C^{\sharp}(u)
    \Bigr) \delta \Lambda_C^{\sharp}(u) \delta \Lambda_T^{\sharp}(u) \\
    (C) & =   \sum_{0 < u \leq \tau} Y^{\sharp}(u)  \Bigl( \delta \Lambda_C^{\sharp}(u)
    \Bigr)^2 \delta \Lambda_T^{\sharp}(u).
      \end{align*}
Because $\delta \Lambda_C^{\dagger}(u)
     -\delta \Lambda_C^{\sharp}(u) \geq 0$
     for $u > 0$, each of $(A)$,
     $(B)$, and $(C)$ are nonnegative.
     Moreover,
\cref{cor:lambda-sharp-coerc} 
implies
that $(A) = (B) = 0$ almost surely if and only if
 \eqref{eq:no-jumps} holds. Similarly,
 $(C) = 0$ almost surely if and only if
 \eqref{eq:no-jumps} holds. Hence, the desired result is proved.
\end{proof}

\section{Additional Material for \texorpdfstring{\cref{sec:apps}}{Section 5}}
\label{sec:supp:km-addntl}

\begin{proof}[Proof of \cref{prop:km-var}]
    The influence function of the Kaplan Meier estimator $F^{\sharp}_n(t)$ is
    \begin{equation*}
        D(X,\Delta; t, \P)
        := \frac{I(X>t)}{G^{\dagger}(t)} + \int_{(0,t]} \frac{F^{\sharp}(t)}{F^{\sharp}(u)} \frac{\d M_C^{\dagger}(u)}{G^{\dagger}(u)} - F^{\sharp}(t),
    \end{equation*}
    as shown in \citet{bakry1994lectures, van2011targeted}. Denote $\V$ as the variance and the covariance operator. Therefore the asymptotic variance of the $n^{1/2}$-scaled Kaplan Meier estimator is 
    \begin{equation}
        \V \left\{ \frac{I(X>t)}{G^{\dagger}(t)} \right\} 
        + \V \left\{ \int_{(0,t]} \frac{F^{\sharp}(t)}{F^{\sharp}(u)} \frac{\d M_C^{\dagger}(u)}{G^{\dagger}(u)} \right\} 
        + 2 \cov \left\{ \frac{I(X>t)}{G^{\dagger}(t)}, \int_{(0,t]} \frac{F^{\sharp}(t)}{F^{\sharp}(u)} \frac{\d M_C^{\dagger}(u)}{G^{\dagger}(u)} \right\}, \label{eq:km-var}
    \end{equation}
    which we simplify in this proof. 

    Recall that $\P(X>t) = F^{\sharp}(t) G^{\dagger}(t)$. The first summand in \cref{eq:km-var} is simply the variance of a Bernoulli random variable and hence satisfies
    \begin{equation*}
        \V \left\{ \frac{I(X>t)}{G^{\dagger}(t)} \right\} 
        = \frac{F^{\sharp}(t) G^{\dagger}(t) \left\{ 1 - F^{\sharp}(t) G^{\dagger}(t) \right\}}{ \{ G^{\dagger}(t) \}^2 }. 
    \end{equation*}
    The last summand in \cref{eq:km-var} may be readily simplified as
    \begin{align*}
        \cov \biggl\{ \frac{I(X>t)}{G^{\dagger}(t)}, \int_{(0,t]} \frac{F^{\sharp}(t)}{F^{\sharp}(u)} & \frac{\d N_C(u) - Y^{\dagger}(u) \,\d \Lambda_C^{\dagger}(u)}{G^{\dagger}(u)} \biggr\} \\
        =\, & \frac{1}{G^{\dagger}(t)} \E \left\{ I(X>t) \int_{(0,t]} \frac{F^{\sharp}(t)}{F^{\sharp}(u)} \frac{\d N_C(u) - Y^{\dagger}(u) \,\d \Lambda_C^{\dagger}(u)}{G^{\dagger}(u)} \right\} \\
        =\, & - \frac{1}{G^{\dagger}(t)} \E \left\{ I(X>t) \int_{(0,t]} \frac{F^{\sharp}(t)}{F^{\sharp}(u)} \frac{\d \Lambda_C^{\dagger}(u)}{G^{\dagger}(u)} \right\} \\
        =\, & - F^{\sharp}(t) \int_{(0,t]} \frac{F^{\sharp}(t)}{F^{\sharp}(u)} \frac{\d \Lambda_C^{\dagger}(u)}{G^{\dagger}(u)}.
    \end{align*}
    The first identity leverages that the martingale transform vanishes in expectation from \cref{thm:martingale-transform}, and the next identity leverages the relationship between the indicator being for the event $X>t$ while the martingale transform is over the domain with $u \leq t$. 

    Now, simplifying the second summand in \cref{eq:km-var} requires appealing to \cref{thm:covar-proc} so that
    \begin{align*}
        \V \left\{ \int_{(0,t]} \frac{F^{\sharp}(t)}{F^{\sharp}(u)} \frac{\d M_C^{\dagger}(u)}{G^{\dagger}(u)} \right\}
        & = \E \left[ \left\{ \int_{(0,t]} \frac{F^{\sharp}(t)}{F^{\sharp}(u)} \frac{\d M_C^{\dagger}(u)}{G^{\dagger}(u)} \right\}^2 \right] \\
        & \hspace{-15mm} = \E \left[ \int_{(0,t]} \left\{ \frac{F^{\sharp}(t)}{F^{\sharp}(u) G^{\dagger}(u)} \right\}^2 \frac{G^{\dagger}(u)}{G^{\dagger}(u-)} Y^{\dagger}(u) \,\d \Lambda_C^{\dagger}(u) \right] \\
        & \hspace{-15mm} = \int_{(0,t]} \left\{ \frac{F^{\sharp}(t)}{F^{\sharp}(u) G^{\dagger}(u)} \right\}^2 \frac{G^{\dagger}(u)}{G^{\dagger}(u-)} G^{\dagger}(u-) F^{\sharp}(u) \,\d \Lambda_C^{\dagger}(u) \\
        & \hspace{-15mm} = \int_{(0,t]} \frac{\{ F^{\sharp}(t) \}^2 }{F^{\sharp}(u)} \frac{\d \Lambda_C^{\dagger}(u)}{G^{\dagger}(u)}.
    \end{align*}

    In summary, we have shown that
    \begin{align*}
        \V \{ D(X,\Delta; t, \P) \}
        & = \frac{F^{\sharp}(t) G^{\dagger}(t) \left\{ 1 - F^{\sharp}(t) G^{\dagger}(t) \right\}}{ \{ G^{\dagger}(t) \}^2 } - \int_{(0,t]} \frac{\{ F^{\sharp}(t) \}^2 }{F^{\sharp}(u)} \frac{\d \Lambda_C^{\dagger}(u)}{G^{\dagger}(u)} \\
        & \hspace{-15mm} = \{ F^{\sharp}(t) \}^2 \left[ \frac{1}{F^{\sharp}(t) G^{\dagger}(t)} 
        - 1 
        - \int_{(0,t]} \frac{1}{F^{\sharp}(u)} \,\d \left\{ \frac{1}{G^{\dagger}(u)} \right\} \right] \\
        & \hspace{-15mm} = \{ F^{\sharp}(t) \}^2 \int_{(0,t]} \frac{1}{G^{\dagger}(u-)} \,\d \left\{ \frac{1}{F^{\sharp}(u)} \right\},
    \end{align*}
    by integration by parts \citep[Theorem A.1.2]{fleming1991counting}. 
\end{proof}

\begin{proof}[Proof of \cref{prop:cox-km-covar}]
    It follows from traditional semiparametric theory that the asymptotic covariance is
    \begin{equation}
        \sigma^2(s,t)
        = \V \left\{ \frac{I(X>s)}{G^{\dagger}(s)} + \int_{(0,s]} \frac{F^{\sharp}(s)}{F^{\sharp}(u)} \frac{\d M_C^{\dagger}(u)}{G^{\dagger}(u)} - F^{\sharp}(s), - \int_{(0,t]} \frac{F^{\sharp}(t)}{F^{\sharp}(u)} \frac{\d M_T^{\sharp}(u)}{G^{\dagger}(u-)} \right\}, \label{eq:km-covar-prf}
    \end{equation}
    where we recall that $\V$ is the covariance operator. This expression can be decomposed into three covariances, and we evaluate them each in turn.
    
    The first expression in \cref{eq:km-covar-prf} is
    \begin{align*}
        \V \biggl\{ \frac{I(X>s)}{G^{\dagger}(s)}, \int_{(0,t]} & \frac{F^{\sharp}(t)}{F^{\sharp}(u)} \frac{\d M_T^{\sharp}(u)}{G^{\dagger}(u-)} \biggr\}
        = \frac{F^{\sharp}(t)}{G^{\dagger}(s)} \E \left\{ I(X>s) \int_{(0,t]} \frac{\d N_T(u) - Y^{\sharp}(u) \,\d \Lambda_T^{\sharp}(u)}{F^{\sharp}(u) G^{\dagger}(u-)} \right\} \\
        & = \frac{F^{\sharp}(t)}{G^{\dagger}(s)} \E \left\{ - I(X > s) \int_{(0,s\wedge t]} \frac{\d \Lambda_T^{\sharp}(u)}{F^{\sharp}(u) G^{\dagger}(u-)} + \int_{(s,t]} \frac{\d N_T(u) - Y^{\sharp}(u) \,\d \Lambda_T^{\sharp}(u)}{F^{\sharp}(u) G^{\dagger}(u-)} \right\} \\
        & = - F^{\sharp}(s) F^{\sharp}(t) \int_{(0,s \wedge t]} \frac{\d \Lambda_T^{\sharp}(u)}{F^{\sharp}(u) G^{\dagger}(u-)} + \frac{F^{\sharp}(t)}{G^{\dagger}(s)} \E \left\{ \int_{(s,t]} \frac{\d N_T(u) - Y^{\sharp}(u) \,\d \Lambda_T^{\sharp}(u)}{F^{\sharp}(u) G^{\dagger}(u-)} \right\} \\
        & = - F^{\sharp}(s) F^{\sharp}(t) \int_{(0,s \wedge t]} \frac{\d \Lambda_T^{\sharp}(u)}{F^{\sharp}(u) G^{\dagger}(u-)}. 
    \end{align*}
    The second identity follows by decomposing the set $(0,t] = (0, s \wedge t] \cup (s, t]$.

    The second expression in \cref{eq:km-covar-prf} is
    \begin{align*}
        \V \biggl\{ \int_{(0,s]} \frac{F^{\sharp}(s)}{F^{\sharp}(u)} \frac{\d M_C^{\dagger}(u)}{G^{\dagger}(u)}, & \int_{(0,t]} \frac{F^{\sharp}(t)}{F^{\sharp}(u)} \frac{\d M_T^{\sharp}(u)}{G^{\dagger}(u-)} \biggr\}
        = \E \left\{ \int_{(0,s]} \frac{F^{\sharp}(s)}{F^{\sharp}(u)} \frac{\d M_C^{\dagger}(u)}{G^{\dagger}(u)} \int_{(0,t]} \frac{F^{\sharp}(t)}{F^{\sharp}(u)} \frac{\d M_T^{\sharp}(u)}{G^{\dagger}(u-)} \right\} 
        = 0,
    \end{align*}
    by applying \Cref{thm:covar-proc}.

    The last expression in \cref{eq:km-covar-prf} also vanishes as it is the expectation of a martingale transform. 
    Therefore the asymptotic covariance is
    \begin{equation*}
        F^{\sharp}(s) F^{\sharp}(t) \int_{(0, s \wedge t]} \frac{\d \Lambda_T^{\sharp}(u)}{F^{\sharp}(u) G^{\dagger}(u-)} 
        = F^{\sharp}(s) F^{\sharp}(t) \int_{(0, s \wedge t]} \frac{1}{G^{\dagger}(u-)} \,\d \left\{ \frac{1}{F^{\sharp}(u)} \right\}. \label{eq:covar-almost-simpl}
    \end{equation*}

\end{proof}

\end{document}